\documentclass{scrartcl}
\usepackage[margin=3cm]{geometry}
\pdfoutput=1 
\usepackage[T1]{fontenc}
\addtokomafont{disposition}{\normalfont\bfseries}
\usepackage{opensans}
\usepackage[british]{babel}
\usepackage[
  maxnames = 4,
  maxalphanames = 4,
  style = alphabetic,
]{biblatex}

\usepackage{amsmath,amssymb}
\usepackage{macros}
\usepackage{style}
\usepackage{authblk}
\author{Jonathan Davies}
\affil{\small School of Mathematical Sciences, University of Nottingham,\\
University Park, Nottingham NG7 2RD, United Kingdom.\vspace{2mm}\\
\href{mailto:ppyjed@nottingham.ac.uk}{\texttt{ppyjed@nottingham.ac.uk}}}
\newcommand{\firstpagefoot}{
\,\,
\emph{Date:} 5 February 2026 \\
\emph{2020 Mathematics Subject Classification:} 18N10, 18A50, 18D20 \\
\emph{Key words and phrases:} 2-category, graded category, enriched category, module category
}
\date{}
\title{Categories graded by group homomorphisms}
\hypersetup{
 pdfauthor={},
 pdftitle={Categories graded by group homomorphisms},
 pdfkeywords={},
 pdfsubject={},
 pdfcreator={},
 pdflang={English}}
\begin{document}

\maketitle
\vspace*{-2em}
\begin{abstract}
We generalise to a group homomorphism \(\tau\) the \(\chi\)-graded categories of Sözer and Virelizier.
These are categories in which both morphisms and objects have compatible degrees.
We give a `half-enriched' Yoneda lemma, a structure theorem for semisimple \(\tau\)-graded categories, and an alternative picture of \(\tau\)-graded categories in terms of pseudofunctors into \(\Cat\).
\nofootnote{\firstpagefoot}
\end{abstract}

\setcounter{tocdepth}{2}
\tableofcontents
\section{Introduction}
\label{sec:orge2d15e9}

Homotopy quantum field theories (HQFTs, see \cite{tv2012}) generalise topological quantum field theories (TQFTs) by allowing manifolds to come equipped with maps to some fixed target space.
The well-known Turaev-Viro state-sum construction builds 3-dimensional TQFTs from spherical fusion categories, and more recent work has extended this construction to HQFTs over certain nice target spaces.
The extended constructions take source categories as follows, writing \(G\) for a group and \(\chi\) for a crossed module.
\begin{center}
\begin{tabular}{lcr}
\toprule
Target space & Source for the 3D state-sum construction & References\\
\midrule
contractible\footnotemark & spherical fusion categories & \cite{tv1992}, \cite{bw1996}\\
connected 1-type & \(G\)-graded spherical fusion categories & \cite{tv2012}\\
connected 2-type & \(\chi\)-graded spherical fusion categories & \cite{sv2023}.\\
\bottomrule
\end{tabular}
\end{center}\footnotetext[1]{\label{org3a50771}When the target space is contractible, HQFTs are the same as TQFTs.}
More examples of \(\chi\)-graded categories are given in \cite{sv2025}.
It is illustrative to remember that connected homotopy \(n\)-types correspond to \(n\)-groups under the homotopy hypothesis, and connected 2-types can be modelled by crossed modules (see Section 1.5 of \cite{haugseng2024}).
A crossed module is a group homomorphism together with a compatible group action, but \(\chi\)-graded categories only use this group action for the monoidal structure.

We do not consider any monoidal structure in this paper, so we are free to define \(\tau\)-graded categories for a group homomorphism \(\tau \colon H \to G\) (\cref{def:tgrad}).
Here \(H\) gives degrees to morphisms and \(G\) to objects, so that an ordinary \(G\)-graded category is just a \(\tau\)-graded category with \(H=1\).
When \(G=1\), \(\tau\)-grading is the same as enrichment over the monoidal category \(\Mod[R]_H\) of \(H\)-graded \(R\)-modules (for \(R\) a fixed non-zero commutative ring).
We also introduce \(\tau\)-graded functors and natural transformations in order to define a 2-category \(\Cat_{\tau}\) of \(\tau\)-graded categories.

The standard enriched Yoneda lemma (as in \cite{kelly2005}) requires a symmetric braiding on the base of enrichment, but \(\Mod[R]_H\) lacks any braiding at all unless \(H\) is abelian.
Therefore we prove from scratch the Yoneda-style \cref{thm:homgradyoneda}.
For each element of \(H\) we end up with a Yoneda embedding of the form
\begin{equation*}
\yo^a_X \colon Y \mapsto \bigoplus_{h \in H} \Hom^{ha} (X,Y) \,,
\end{equation*}
`half-enriched' so that each \(\yo^a_X\) is \(\Mod[R]_H\)-enriched but \(\yo_{(-)}^a\) is just \(R\)-linear.
We then define shift objects, which represent these embeddings with respect to one another.
Shift objects are additional structure on \(\tau\)-graded categories, but we later show in \cref{cor:additive-hom-graded-def,prop:sstau} that they are automatically present in all additive or semisimple \(\tau\)-graded categories (and hence also in \(\chi\)-graded fusion categories).
Shift objects give ways to view the entirety of a \(\tau\)-graded category from the perspective of its degree 1 morphisms, and will be used throughout the rest of the paper.

Now that we have introduced all the basic concepts, \cref{sec:semisimplicity} and \cref{sec:2eqv} are broadly independent of one another.
\Cref{sec:semisimplicity} deals with the special case of semisimple \(\tau\)-graded categories.
We recall their definition as in \cite{sv2023} and give an equivalent definition in terms of shifts.
The rest of the section works towards a structure theorem, giving a classification of semisimple \(\tau\)-graded categories.
In \cref{def:mtau}, we define a class of indecomposable semisimple \(\tau\)-graded categories of the form \(\cat{M}_{\tau}(L,\psi)\langle g \rangle\), where \(L\) is a subgroup of \(\ker(\tau)\), \(\psi\) is a normalised 2-cocycle \(H^2 \to \Hom_{\Set}(H/L,R^{\times})\) and \(g\) is an element of \(G\).

\begin{introtheorem}[\cref{thm:tgrad:struct}]
A \(\tau\)-graded category \(\cat{C}\) is semisimple if and only if there is a \(\tau\)-graded equivalence
\begin{equation*}
\cat{C} \eqv \bigboxplus_{i \in I} \cat{M}_{\tau}(L_i,\psi_i)\langle g_i \rangle
\end{equation*}
for some subgroups \(L_i \le \ker(\tau)\), 2-cochains \(\psi_i \colon H^2 \to \Hom_{\Set}(H/L_i,R^{\times})\) and elements \(g_i \in G\).
We can choose \((g_i)_{i \in I}\) to be the degrees of a repesentative collection of simple objects.
\label{thm:A}
\end{introtheorem}

\begin{introtheorem}[\cref{thm:mtau-equiv}]
Each \(\tau\)-graded equivalence
\begin{equation*}
\cat{M}_{\tau}(L,\psi)\langle g \rangle \to \cat{M}_{\tau}(L',\psi')\langle g' \rangle
\end{equation*}
corresponds to a choice of group element and normalised 1-cochain
\begin{equation*}
t \in \tau^{-1}(gg'^{-1})
\qquad
\gamma \colon H \to \Hom_{\Set}(H/L,R^{\times})
\end{equation*}
with \(L = tL't^{-1}\) and \(\psi = \psi'^t \cdot d\gamma\), writing \(\psi'^{t} (a,b) (hL') = \psi' (a,b) (thL')\).
Natural isomorphisms \(F_{t,\gamma} \tonat F_{s,\delta}\) require \(tL' = sL'\) and correspond to 0-cochains \(\eta \colon H/L \to R^{\times}\) with \(\gamma = \delta \cdot d\eta\).
\label{thm:B}
\end{introtheorem}
Both of these theorems are proved by constructing explicit \(\tau\)-graded equivalence functors.
We also detail how to construct the groups \(L_i\) and 2-cochains \(\psi_i\).
These results are reminiscent of those in \cite{ostrik2003} and \cite{natale2017} for semisimple module categories.
One can also compare the use of cochains in Theorem 8.3.7 of \cite{bl2004} for classifying the functors and natural transformations between 2-groups.
However, in that case the source and target are associated with two different pairs \((H,G) \to (H',G')\); here we have the same groups (mediated by \(\tau\)) and allow additional data to vary.

We investigate the relationship between module categories and \(\tau\)-graded categories in \cref{sec:2eqv}.
When a \(\tau\)-graded category has shifts, we show that that its subcategory of degree 1 morphisms becomes an \(H\)-module category.\footnote{Here as in \cite{egno2015}, an \(H\)-module category is an \(R\)-linear category \(\cat{C}\) together with a monoidal functor \(\Disc(H) \to \Aut(\cat{C})\), where \(\Disc(H)\) is \(H\) viewed as a discrete category.}
In the special case \(G=1\), we extend this to a 2-functor between the 2-category of \(H\)-module categories and the 2-category of \(\tau\)-graded categories with shifts (\cref{thm:2eqv-H}).

Each group homomorphism induces a canonical groupoid \(\cat{G}_{\tau}\) as in \cref{eg:tgrad:grpd}.
We introduce \(\tau\)-module categories, which are pseudofunctors into \(\Cat\) from the discrete categorification of \(\cat{G}_{\tau}\).
These extend \(H\)-module categories in the case \(G \neq 1\), and the final theorem is obtained by extending \cref{thm:2eqv-H}.

\begin{introtheorem}[\cref{thm:final2eqv}]
There is a 2-equivalence
\begin{equation*}
\Cat_{\tau}^{\shifts} \eqv \ModCat[H]_{\tau} \,,
\end{equation*}
where \(\Cat_{\tau}^{\shifts}\) is the 2-category of \(\tau\)-graded categories with shifts and \(\ModCat[H]_{\tau}\) is the 2-category of \(\tau\)-module categories.
\end{introtheorem}
This further restricts to a 2-equivalence between semisimple \(\tau\)-graded and \(\tau\)-module categories.
When \(G=1\), the 2-equivalence precisely translates between \cref{thm:A,thm:B} and the results of \cite{ostrik2003} and \cite{natale2017}.

When \(\chi=\tau\) is a crossed module, we \(\Cat_{\chi}\) admits a monoidal 2-category structure exhibiting the \(\chi\)-graded monoidal categories as algebra objects.
We will focus on this in a future paper.
\subsection{Conventions}
\label{sec:conventions}
Here we introduce some conventions and notation which will be used throughout the paper.
We fix a non-zero commutative ring \(R\), arbitrary groups \(G\) and \(H\), and a group homomorphism \(\tau \colon H \to G\).
Our conventions for enriched categories are from \cite{kelly2005}, and for 2-categories and pseudofunctors from \cite{jy2021} (in particular, we take 2-categories to be strict bicategories).

We call a category \defined{\(R\)-linear} if it is enriched over the category \(\Mod[R]\) of \(R\)-modules.
The \(R\)-linearisation of a set \(X\) is the free \(R\)-module \(RX\) whose elements are finite \(R\)-linear combinations of those in \(X\).
Inductively, the \(R\)-linearisation of an \(n\)-category \(\cat{C}\) is the category \(R\cat{C}\) with the same objects and each \(\Hom_{R\cat{C}}(X,Y) = R\Hom_{\cat{C}}(X,Y)\).

We call a category \defined{additive} if it is enriched over the category \(\Ab\) of abelian groups and has all finite direct sums.
For each category \(\cat{C}\), we can add formal finite direct sums to give the additive completion \(\cat{C}^{\oplus}\); if \(\cat{C}\) is already additive then we have \(\cat{C}^{\oplus} \eqv \cat{C}\).

For any \(n\)-category \(\cat{C}\) (including \(n=0\) for sets), we write
\begin{itemize}
\item \(\deloop\cat{C}\) for the delooping, that is the \((n+1)\)-category with a single object \(\bullet\) whose endomorphism category is \(\cat{C}\)
\item \(\Disc(\cat{C})\) for the discrete categorification, that is \(\cat{C}\) viewed as an \((n+1)\)-category with only trivial \((n+1)\)-morphisms.
\end{itemize}
In all categories, we assume the axiom of choice so that we can work with representatives of equivalence classes and so that functors are equivalences if and only if they are fully faithful and essentially surjective.
We use a double arrow (\(\tonat\)) to describe 2-morphisms inline or in a diagram of 1-morphisms; otherwise we always use a single arrow (\(\to\)).
We write \(\hof\) for horizontal composition (along 1-morphisms) and \(\of\) for vertical composition (along objects).
\section{Categories graded by a group homomorphism}
\label{sec:general}
We introduce the 2-category \(\Cat_{\tau}\) of \(\tau\)-graded categories along with various important sub-2-categories and properties.
We prove a Yoneda-style lemma for \(\Cat_{\tau}\) in \cref{sec:yoneda} and introduce shift objects in \cref{sec:shifts}.
Shift objects will remain a central idea throughout the rest of the paper.
\subsection{Definitions}
\label{sec:defs}
We recall from \cite{sv2023} the definitions of \(H\)-Hom-graded and \(G\)-graded categories.

\begin{itemize}
\item An \(R\)-linear category \(\cat{C}\) is \defined{\(H\)-Hom-graded} if it is enriched over the category \(\Mod[R]_H\) of \(H\)-graded \(R\)-modules.
That is, if each homspace is an \(H\)-graded \(R\)-module, each identity morphism is homogeneous of degree 1, and composition arises from \(R\)-linear maps
\begin{equation*}
\of \colon
\Hom_{\cat{C}}^{h'}(Y,Z) \otimes_R \Hom_{\cat{C}}^h(X,Y)
\to
\Hom_{\cat{C}}^{h'h}(X,Z)
\,.
\end{equation*}
We write \(\cat{C}^1\) for the \(R\)-linear subcategory containing only degree 1 morphisms, and we write \(X \iso Y\) in \(\cat{C}\) if and only if \(X \iso Y\) in \(\cat{C}^1\).

\item A category \(\cat{D}\) is \defined{\(G\)-graded} if there are disjoint subcategories \((\cat{D}_g)_{g \in G}\) such that each object \(X \in \cat{D}\) has a decomposition \(\bigoplus_{g \in G} X_g\) with finitely many non-zero components \(X_g \in \cat{D}_g\).
\end{itemize}

In either case, we write \(|\cdot|\) for the map sending homogeneous objects or morphisms to their degrees.
We write \(\Cat^H\) for 2-category of \(H\)-Hom-graded categories, following the standard construction for enriched categories in \cite{kelly2005}; in particular, \(H\)-Hom-graded functors preserve morphism degrees and \(H\)-Hom-graded natural transformations only have components of degree 1.
For \(\cat{C},\cat{D} \in \Cat^H\), we also write
\begin{align*}
\Fun^H(-,-) &= \Hom_{\Cat^H}(-,-) &
\Nat^H(-,-) &= \Hom_{\Fun^H(\cat{C},\cat{D})}(-,-) \,.
\end{align*}
We write \(\Cat_G\) for the 2-category of \(R\)-linear \(G\)-graded categories, where 1-morphisms preserve object degrees.

In \cite{sv2023}, Sözer and Virelizier define monoidal categories graded by crossed modules.
A crossed module is a group homomorphism with extra structure, necessary for the monoidal coherences in such categories.
In this paper, we do not consider any monoidal structure and hence it is sufficient to work over an ordinary group homomorphism.
The below definition of \(\tau\)-graded categories is precisely this generalisation.

\begin{definition}
Let \(\tau \colon H \to G\) be a group homomorphism.
An \(R\)-linear category \(\cat{C}\) is called \defined{\(\tau\)-graded} if
\begin{enumerate}
\item \(\cat{C}\) is \(H\)-Hom-graded
\item \(\cat{C}^1\) is \(G\)-graded
\item for \(X \in \cat{C}_x\) and \(Y \in \cat{C}_y\), we have \(\Hom_{\cat{C}}^h(X,Y) = 0\) unless \(y = \tau(h)x\).
\end{enumerate}
We write \(\Cat_{\tau}\) for the 2-category of \(\tau\)-graded categories, where
\begin{itemize}
\item \(\tau\)-graded functors are \(H\)-Hom-graded functors \(F\) satisfying \(|FX|=|X|\) for all homogeneous objects \(X\).
\item \(\tau\)-graded natural transformations are \(H\)-Hom-graded natural transformations with components of degree 1.
\end{itemize}
For \(\cat{C},\cat{D} \in \Cat^{\tau}\), we also write
\begin{align*}
\Fun_{\tau}(-,-)
&= \Hom_{\Cat_{\tau}}(-,-)
&
\Nat_{\tau}(-,-)
&= \Hom_{\Fun_{\tau}(\cat{C},\cat{D})}(-,-)
\,.
\end{align*}
\label{def:tgrad}
\end{definition}

\begin{definition}
We write \((-)^1\) for the strictly surjective 2-functor \(\Cat_{\tau} \toepi \Cat_G\) mapping
\begin{itemize}
\item each \(\tau\)-graded category \(\cat{C}^1\) to the \(R\)-linear subcategory \(\cat{C}^1\) containing only the degree 1 morphisms
\item 1- and 2-morphisms to the appropriate restrictions (denoted \(F \mapsto F^1\) and \(\eta \mapsto \eta^1\) respectively).
\end{itemize}
\end{definition}
\begin{remark}
There are canonical embeddings
\begin{equation*}
\Cat_G \tomono \Cat_{\tau} \otmono \Cat^H \,.
\end{equation*}
The embedding of \(\Cat_G\) adds only the zero morphisms of non-trivial degree. The 2-functor \((-)^1\) provides a retraction, giving a 2-equivalence when \(H=1\).
On the other hand, the embedding of \(\Cat^H\) assigns degree 1 to each object.
It has a retraction by forgetting object degrees, giving a 2-equivalence when \(G=1\).
\end{remark}

The following examples will appear frequently in later sections.
\Cref{eg:tgrad:grpd,eg:tgrad:modtau} follow from forgetting the monoidal structure in Examples 4.4 and 4.5 of \cite{sv2023}.
\Cref{eg:tgrad:cyclic} is new to this paper.

\begin{eg}[\(R\cat{G}_{\tau}\)]
Each group homomorphism \(\tau \colon H \to G\) induces a canonical groupoid \(\cat{G}_{\tau}\).
Objects are elements \(g \in G\), and morphisms are pairs \((h,g) \colon g \to \tau(h)g\) for \(g \in G, h \in H\).
The \(R\)-linearisation \(R\cat{G}_{\tau}\) is \(\tau\)-graded with each \(|g| = g\) and \(|(h,g)| = h\).
\label{eg:tgrad:grpd}
\end{eg}

\begin{eg}[\({\Mod[R]_{\tau}}\) and \({\Mod[R]_H^{\bullet}}\)]
Recall that we write \(\Mod[R]_H\) for the category of \(H\)-graded \(R\)-modules.
This has no braiding unless \(H\) is abelian.
It is however still biclosed monoidal, with right and left internal homs
\begin{align*}
[ M, N ]
&= \bigoplus_{h \in H} \{ f \in \Hom_R(M,N) \st f(M_a) \subseteq N_{ha} \ \forall a \in H \}
\\
\langle M, N \rangle
&= \bigoplus_{h \in H} \{ f \in \Hom_R(M,N) \st f(M_a) \subseteq N_{ah} \ \forall a \in H \}
\end{align*}
acting by the usual pre- and post-composition on morphisms.
For an arbitrary group homomorphism \(\tau \colon H \to G\), we define the \(\tau\)-graded category \(\Mod[R]_{\tau}\) of \(G\)-graded \(R\)-modules with
\begin{equation*}
\Hom_{\Mod[R]_{\tau}}^h (M,N)
= \{ f \in \Hom_R(M,N) \st
f(M_g) \subseteq N_{\tau(h)g} \ \forall g \in G
\} \,.
\end{equation*}
Writing \(\Mod[R]_H^{\bullet} = \Mod[R]_{\id[H]}\), we have \(\Hom_{\Mod[R]_H^{\bullet}}(M,N) = [M,N]\) so that we can consider \(\Mod[R]_H^{\bullet}\) to be the right-enrichment of \(\Mod[R]_H\) over itself.
We also have \(\Mod[R]_{\tau}^1 = \Mod[R]_H\).
\label{eg:tgrad:modtau}
\end{eg}

\begin{eg}[\(\tau \colon C_8 \to C_2\)]
Let \(H = C_8 = \langle x \rangle\) and \(G = C_2 = \langle y \rangle\).
There is a single non-trivial group homomorphism \(\tau \colon H \to G\), which acts by \(x \mapsto y\).
For such \(\tau\), we define four different \(\tau\)-graded categories in the table below.
Homspace components are assumed to be zero unless otherwise specified, and composition is by multiplication in \(R\).
\begin{center}
\begin{tabular}{c||c|c||c}
\toprule
Category & Objects & \(\Hom^{x^n}(a,b) = R\) when & \(\dim_R\big( \Hom(a,b) \big)\)\\
\midrule
\(\cat{C}_1\) & \(\{0,1,2,3,4,5,6,7\}\) & \(a+n \equiv b \pmod{8}\) & 1\\
\(\cat{C}_2\) & \(\{0,1,2,3\}\) & \(a+n \equiv b \pmod{4}\) & 2\\
\(\cat{C}_4\) & \(\{0,1\}\) & \(a+n \equiv b \pmod{2}\) & 4\\
\(\cat{C}_8\) & \(\{0\}\) & \(a+n = b\) & 8\\
\bottomrule
\end{tabular}
\end{center}
Each object \(a\) has degree \(|a| = \tau(x^a) = y^a\).
In any of the above categories, \(\Hom^{x^n}(a,b) \neq 0\) then implies the \(\tau\)-grading condition
\begin{equation*}
\tau(x^n) + |a|
= y^n + y^a
= y^{a+n}
= y^b
= |b|
\,.
\end{equation*}
The object grading is trivial for \(\cat{C}_8\) and the homspace grading is trivial for \(\cat{C}_1\).
The categories \(\cat{C}_2\) and \(\cat{C}_4\) lie somewhere in between.
\label{eg:tgrad:cyclic}
\end{eg}
\subsection{Yoneda lemma}
\label{sec:yoneda}
The standard \(\cat{V}\)-enriched Yoneda embedding as in \cite{kelly2005} is of the form
\begin{equation*}
\cat{C} \tomono \Hom_{\cat{V}\dash\Cat}(\cat{C}^{\op},\cat{V}) \,.
\end{equation*}
However, when we try to set \(\cat{V} = \Mod[R]_H\) we encounter a sequence of issues:
\begin{enumerate}
\item \(\cat{C}^{\op}\) is not well-defined without a braiding on \(\Mod[R]_H\)
\item \(\Fun^H(\cat{C},\Mod[R]_H)\) is not well-defined because \(\Mod[R]_H\) is not itself Hom-graded
\item \(\Fun^H(\cat{C},\Mod[R]_H^{\bullet})\) is only trivially \(H\)-Hom-graded, since Hom-graded natural transformations must have degree 1.
\end{enumerate}
We instead define `half-enriched' contravariant Yoneda embeddings of the form
\begin{equation*}
(\cat{C}^1)^{\op}
\tomono
\Fun^H(\cat{C},\Mod[R]_H^{\bullet})
\,.
\end{equation*}
In fact, we obtain a collection of Yoneda embeddings indexed over \(H\).
We use the Japanese symbol \(\yo\) (pronounced `yo') following the notation of \cite{js2017}.

\begin{proposition}
Let \(\cat{C}\) be an \(H\)-Hom-graded category.
Then for each \(a \in H\) and \(X \in \cat{C}\) there is an \(H\)-Hom-graded functor \(\yo_X^a \colon \cat{C} \to \Mod[R]_H^{\bullet}\) acting on objects and morphisms respectively by
\begin{equation*}
Y \mapsto \bigoplus_{h \in H} \Hom_{\cat{C}}^{ha}(X,Y) \qquad
y \mapsto y \of (-) \,.
\end{equation*}
Moreover, for each \(a \in H\) there is an \(R\)-linear functor \(\yo^a \colon (\cat{C}^1)^{\op} \to \Fun^H(\cat{C}, \Mod[R]_H^{\bullet})\)  mapping objects \(X \mapsto \yo_X^a\) and morphisms \(x \mapsto (-) \of x\).
\label{prop:def:homgradyoneda}
\end{proposition}
\begin{remark}
When \(a=1\) and \(\cat{C} = \Mod[R]_H\) (viewed as an \(H\)-Hom-graded category with only degree 1 morphisms), \(\yo^a\) is precisely the right internal hom-functor as in \cref{eg:tgrad:modtau}.
\end{remark}
\begin{proof}
Fix homogeneous morphisms
\begin{align*}
y &\in \Hom_{\cat{C}}^k(Y,Y') \qquad &
y' &\in \Hom_{\cat{C}}^{k'}(Y',Y'') \\
x &\in \Hom_{\cat{C}}^1(X,X') &
x' &\in \Hom_{\cat{C}}^1(X',X'')
\end{align*}
and \(f \in (\yo_X^a Y)_h = \Hom_{\cat{C}}^{ha}(X,Y)\).
Well-definedness of \(\yo_X^a\) is by
\begin{equation*}
(\yo_X^a y)(f)
= y \of f
\in \Hom_{\cat{C}}^{kha}(X,Y')
= (\yo_X^a Y')_{kh}
\end{equation*}
so that \(\yo_X^a y\) is indeed a degree \(k\) morphism \(\yo_X^a Y \to \yo_X^a Y'\).
Well-definedness of \(\yo^a\) follows from the associativity of composition, which gives the commuting naturality square
\begin{equation*}
\begin{tikzcd}[column sep=5em]
(\yo_X^a Y)_h \ar[r,"\yo_x^a = (-) \of x"] \ar[d,swap,"\yo_x^a y = y \of (-)"]
&
(\yo_{X'}^a Y)_h \ar[d,"\yo_{x'}^a y = y \of (-)"]
\\
(\yo_X^a Y')_{kh} \ar[r,swap,"\yo_x^a = (-) \of x"]
&
(\yo_{X'}^a Y')_{kh} \,.
\end{tikzcd}
\end{equation*}
Functoriality is then by
\begin{align*}
\yo_X^a \id[Y]
&= \id[Y] \of (-)
= \id[\yo_X^a Y]
&
\yo_{\id}^a
&= (-) \of \id
= \id
\\
\yo_X^a (y' \of y)
&= y' \of y \of (-)
= \yo_X^a y' \of \yo_X^a y
&
\yo_{x' \of x}^a
&= (-) \of x' \of x
= \yo_x^a \yo_{x'}^a
\,.\qedhere
\end{align*}
\end{proof}

\begin{theorem}[Yoneda lemma for Hom-graded categories]
Let \(\cat{C}\) be an \(H\)-Hom-graded category and fix \(a \in H\).
Then for each \(a \in H\), \(X \in \cat{C}\) and \(F \in \Fun^H(\cat{C},\Mod[R]_H^{\bullet})\) there is an isomorphism
\begin{equation*}
\Phi^a_{X,F}
\colon \Nat^H(\yo_X^a,F) \to (FX)_{a^{-1}}
\end{equation*}
in \(\Mod[R]\), where
\begin{itemize}
\item \(\Phi^a_{X,F}\) maps \(\eta \colon \yo^a_X \tonat F\) to \(\eta_X(\id[X])\)
\item \((\Phi^a_{X,F})^{-1}\) maps \(v \in (FX)_{a^{-1}}\) to the natural transformation with components
\begin{equation*}
(\Phi^a_{X,F})^{-1} (v)_Y \colon f \mapsto (Ff)(v) \,.
\end{equation*}
\end{itemize}
Moreover this is natural in \(X\) and \(F\), viewing the source and target as \(R\)-linear functors
\begin{equation*}
\cat{C}^1 \otimes_R \Fun^H(\cat{C}, \Mod[R]_H^{\bullet})
\to \Mod[R]
\,.
\end{equation*}
\label{thm:homgradyoneda}
\end{theorem}
\begin{proof}
For each \(X \in \cat{C}\) we have
\begin{equation*}
\id[X]
\in \Hom_{\cat{C}}^1(X,X)
= \Hom_{\cat{C}}^{a^{-1}a}(X,X)
= (\yo_X^a X)_{a^{-1}}
= (\yo_X^a X)_{a^{-1}}
\,.
\end{equation*}
Then for each \(\eta \in \Nat^H(\yo^a_X,F)\) we must also have \(\eta_X(\id[X]) \in (FX)_{a^{-1}}\), hence \(\Phi^a_{X,F}\) is well-defined.

Now fix \(v \in (FX)_{a^{-1}}\) and \(f \in (\yo_X^a Y)_h = \Hom_{\cat{C}}^{ha}(X,Y)\), and write \(\Psi(v)_Y(f) = (Ff)(v)\).
Since \(Ff \in \Hom_{\cat{C}}^{ha}(FX,FY)\), we have
\begin{equation*}
\Psi(v)_Y(f) \in (FY)_{haa^{-1}} = (FY)_h
\end{equation*}
so that \(\Psi(v)_Y\) is degree 1 as a morphism \(\yo_X^a Y \to FY\).
We have \(\Psi = (\Phi_{X,F}^a)^{-1}\) by
\begin{align*}
(\Phi^a_{X,F} \of \Psi)(v)
&= \Psi(v)_X(\id[X])
= (F \id[X])(v)
= \id[FX](v)
= v
\\
\big( (\Psi \of \Phi^a_{X,F})(\eta) \big)_Y (f)
&= (Ff \of \eta_X)(\id[X])
= (\eta_Y \of \yo_X^a f)(\id[X])
= \eta_Y(f)
\,,
\end{align*}
and \(\Psi(v)_Y\) is natural in \(Y\) by
\begin{equation*}
\big( \Psi(v)_{Y'} \of \yo_X^a y \big)(f)
= \Psi(v)_{Y'}(y \of f)
= (Fy \of Ff)(v)
= \big( Fy \of \Psi(v)_Y \big)(f)
\,.
\end{equation*}
Finally, for \(x \in \Hom_{\cat{C}}^1(X,X')\) and \(\alpha \colon F \tonat F'\) we have naturality of \(\Phi\) by
\begin{align*}
(Fx \of \Phi_{X,F}^a)(\eta)
&= (Fx \of \eta_X)(\id[X]) \\
&= (\eta_{X'} \of \yo_X^a x)(\id[X]) \\
&= \eta_{X'}(x) \\
&= (\eta \of \yo_x^a)_{X'}(\id[X']) \\
&= \Phi_{X',F}^a(\eta \of \yo_x^a) \\
&= (\Phi_{X',F}^a \of \Nat^H(\yo_x^a,F))(\eta)
\end{align*}
and
\begin{align*}
(\alpha_X \of \Phi^a_{X,F}) (\eta)
= (\alpha \of \eta)_X (\id[X])
= \Phi^a_{X,F'}(\alpha \of \eta)
= (\Phi^a_{X,F'} \of \Nat^H(\yo_X^a,\alpha))(\eta)
\,.
\end{align*}
\end{proof}

\begin{corollary}
Each \(\yo^a\) is fully faithful.
\end{corollary}
\begin{proof}
From \cref{thm:homgradyoneda}, we consider the \(R\)-linear isomorphism
\begin{equation*}
\Phi^a_{X, \yo^a_Y}
\colon \Nat^H(\yo^a_X, \yo^a_Y)
\to (\yo^a_Y X)_{a^{-1}}
= \Hom_{\cat{C}^1}(Y,X)
\,.
\end{equation*}
For morphisms \(v \in \Hom_{(\cat{C}^1)^{\op}}(X,Y) = \Hom_{\cat{C}^1}(Y,X)\) and \(f \in \Hom_{\cat{C}^1}(Y,Y')\), we have
\begin{equation*}
(\yo_v^a)_{Y'} f
= f \of v
= (\yo^a_Y f)(v)
= (\Phi^a_{X,\yo^a_Y})^{-1}(v)_{Y'} (f)
\,.
\end{equation*}
Since \((\Phi^a_{X,\yo^a_Y})^{-1}\) is invertible, \(\yo^a\) is bijective on homsets and hence fully faithful.
\end{proof}
\subsection{Shift objects}
\label{sec:shifts}
Intuitively, shift objects give a way for \(\cat{C}^1\) to `see' the rest of \(\cat{C}\).
This allows us to easily extend ordinary \(R\)-linear categorical properties such as additivity and semisimplicity.
We will also see in \cref{prop:shiftrepr} that shift objects represent the Hom-graded Yoneda embeddings of \cref{prop:def:homgradyoneda}.

\begin{definition}
Let \(X\) be an object in a \(\tau\)-graded category, and let \(a \in H\).
A \defined{shift} of \(X\) by \(a\), if it exists, is an object \(X\langle a \rangle\) together with a degree \(a\) isomorphism \(r_{X,a} \colon X \to X\langle a \rangle\).
\end{definition}
\begin{remark}
Shifts automatically induce \(R\)-linear isomorphisms
\begin{align*}
(-) \of r_{X,a} &\colon \Hom_{\cat{C}}^h(X\langle a \rangle,Y) \tonat \Hom_{\cat{C}}^{ha}(X,Y) \\
r_{Y,a} \of (-) &\colon \Hom_{\cat{C}}^h(X,Y) \tonat \Hom_{\cat{C}}^{ah}(X,Y\langle a \rangle) \,.
\end{align*}
\end{remark}
\begin{remark}
If \(X\) is homogeneous of degree \(x\), then it follows that any shift \(X\langle a \rangle\) must have degree \(\tau(a)x\).
Given another shift \(X[a]\) of \(X\) by \(a\), it is easy to construct a degree 1 isomorphism \(X[a] \to X\langle a \rangle\).
\label{rem:shiftdegree}
\end{remark}
\begin{proposition}
Let \(X\) and \(Y\) be objects in a \(\tau\)-graded category \(\cat{C}\), and fix \(a \in H\).
For \(b \in H\), observe that \(\yo_Y^b \iso \yo_X^{ba}\) is the same as saying
\begin{equation*}
\bigoplus_{h \in H} \Hom_{\cat{C}}^{hb}(Y,-) \iso \bigoplus_{h \in H} \Hom_{\cat{C}}^{hba}(X,-) \,.
\end{equation*}
The following are equivalent:
\begin{enumerate}
\item there is a degree \(a\) isomorphism \(X \to Y\) making \(Y\) a shift of \(X\)
\item we have \(\yo_Y^b \iso \yo_X^{ba}\) for some \(b \in H\)
\item we have \(\yo_Y^b \iso \yo_X^{ba}\) for all \(b \in H\).
\end{enumerate}
\label{prop:shiftrepr}
\end{proposition}
\begin{proof}
Clearly (3) implies (2).
By the Hom-graded Yoneda lemma (\cref{thm:homgradyoneda}), we have for each \(b \in H\) a correspondence between \(H\)-Hom-graded natural isomorphisms \(\yo_Y^b \tonat \yo_X^{ha}\) and isomorphisms in
\begin{equation*}
(\yo_X^{ha} Y)_{b^{-1}} = \Hom_{\cat{C}^a}(X,Y) \,.
\end{equation*}
Therefore (2) implies (1) and, considering all \(b \in H\), (1) implies (3).
\end{proof}

\begin{definition}
A \(\tau\)-graded category \(\cat{C}\) \defined{has shifts} if for each \(X \in \cat{C}\) and \(a \in H\) there exists a shift object \(X\langle a \rangle \in \cat{C}\).
We write \(\Cat_{\tau}^{\shifts}\) for the sub-2-category of \(\Cat_{\tau}\) containing the \(\tau\)-graded categories with shifts.
\label{def:hasshifts}
\end{definition}
\begin{remark}
Let \(F \colon \cat{C} \to \cat{D}\) be a \(\tau\)-graded functor between \(\tau\)-graded categories with shifts.
We then automatically have degree 1 isomorphisms
\begin{equation*}
Fr_{X,a} \of r_{FX,a}^{-1}
\colon (FX)\langle a \rangle
\to F(X\langle a \rangle)
\,.
\end{equation*}
\end{remark}
\begin{remark}
For any \(\tau\)-graded category \(\cat{C}\), we can construct a \(\tau\)-graded category \(\cat{C}^{\shifts}\) by adding new objects \(X\langle h \rangle\) and morphisms \(r_{X,h} \in \Hom_{\cat{C}^{\shifts}}^h(X,X\langle h \rangle)\) for each \(X \in \cat{C}\) and \(h \in H\).
Since shifts are unique up to degree 1 isomorphism, we have \(\cat{C}^{\shifts} \eqv \cat{C}\) whenever \(\cat{C}\) already has shifts.
\end{remark}
\begin{eg}
Recall the \(\tau\)-graded category \(R\cat{G}_{\tau}\) from \cref{eg:tgrad:grpd}.
Each object \(g \in R\cat{G}_{\tau}\) has shifts given by
\begin{equation*}
g\langle a \rangle = \tau(a)g \qquad
r_{g,a} := (g,a)
\in \Hom_{R\cat{G}_{\tau}}\big( g, \tau(a)g \big)
\,.
\end{equation*}
\label{eg:shifts:grpd}
\end{eg}
\begin{eg}
Recall the \(\tau\)-graded category \(\Mod[R]_{\tau}\) from \cref{eg:tgrad:modtau}.
Each object \(M \in \Mod[R]_{\tau}\) is of the form \(M = \bigoplus_{g \in G} M_g\) for finitely many non-zero \(M_g \in \Mod[R]\).
We can then construct \(M\langle a \rangle = \bigoplus_{g \in G} M_{\tau(a)^{-1}g}\) so that each \(M\langle a \rangle_{\tau(a)g} = M_g\).
Shift isomorphisms \(r_{M,a}\) act by permuting components.
\label{eg:shifts:modtau}
\end{eg}
\begin{eg}
For \(H=C_8=\langle x \rangle\) and \(G=C_2=\langle y \rangle\), recall the \(\tau\)-graded categories \((\cat{C}_k)_{k=1,2,4,8}\) from \cref{eg:tgrad:cyclic}.
Each object \(a\) in \(\cat{C}_k\) has shifts given by
\begin{equation*}
a\langle x^n \rangle = a+n \qquad
r_{a,x^n} = 1 \in \Hom^{x^n}(a,a+n)
\end{equation*}
where addition in \(\ob(\cat{C}_k)\) is defined modulo \(k\).
\label{eg:shifts:cyclic}
\end{eg}
\section{Semisimplicity}
\label{sec:semisimplicity}
We begin by reviewing semisimplicity for \(R\)-linear categories.
Sözer and Virelizier use what we call \(R\)-linear semisimplicity, which often differs from the standard definition of semisimplicity for abelian categories.
We discuss both of these notions in turn and how they relate to each other, and we prove a structure theorem for \(R\)-linear semisimple categories (\cref{thm:objss}).
This structure theorem is well-known, but appears to lack a published proof or explicit statement.

In \cref{sec:additivity} we then discuss additivity for \(\tau\)-graded categories, which Sözer and Virelizier define in terms of what they call \(h\)-direct sums.
We show that their definition is equivalent to requiring shifts and only those \(h\)-direct sums for which \(h=1\).
In \cref{sec:sstau} we interpret semisimplicity for \(\tau\)-graded categories along similar lines.
Finally, we use this to prove a new structure theorem for semisimple \(\tau\)-graded categories in \cref{sec:struct}.
\subsection{Review of semisimplicity for \(R\)-linear categories}
\label{sec:rss}
There are multiple different definitions of semisimple category in the literature.
A helpful discussion on this topic can be found in \cite{bailie2017}, which suggests names for each type.
Abelian semisimplicity is used in \cite{egno2015} and is probably the most standard notion.

On the other hand, \cite{sv2023} uses what \cite{bailie2017} calls object semisimplicity; we instead call it \(R\)-linear semisimplicity to emphasise its dependence on the base ring \(R\).
We warn that these two notions of semisimplicity can be very different; in particular, \(R\)-linear `simple' objects need not be abelian simple even when \(R=\mathbb{C}\).
\subsubsection{Abelian semisimplicity}
\label{sec:org98c8998}

In an abelian category, recall that a \defined{subobject} of an object \(X\) is a monomorphism \(Y \tomono X\).

\begin{definition}
A non-zero object \(X\) in an abelian category is called \defined{abelian simple} if it has no non-trivial subobjects.
A category is called \defined{abelian semisimple} if it is abelian and every object is a finite direct sum of abelian simple objects.
\label{def:abss}
\end{definition}

Schur's lemma for abelian categories is well-known --- see, for example, Lemma 1.5.2 of \cite{egno2015}.
We will later use it to give sufficient conditions for an \(R\)-linear semisimple category to be abelian semisimple.

\begin{lemma}[Schur's lemma]
Let \(\cat{C}\) be an abelian category with abelian simple objects \(S\) and \(T\).
Then \(\End_{\cat{C}}(S)\) is a division ring and every morphism \(S \to T\) is either zero or an isomorphism.
If we have further that \(\cat{C}\) is hom-finite \(\field\)-linear over an algebraically closed field \(\field\), then every abelian simple object \(S \in \cat{C}\) has \(\End_{\cat{C}}(S) \iso \field\).
\qed
\label{lem:schur}
\end{lemma}
\begin{corollary}
Let \(\cat{C}\) be an \(R\)-linear abelian category.
Suppose there is an abelian simple object \(S \in \cat{C}\) with \(\End_{\cat{C}}(S) \iso R\).
Then \(R\) is a field.
\label{cor:abs-obs-implies-field}
\end{corollary}
\begin{proof}
Schur's lemma (\cref{lem:schur}) makes \(\End_{\cat{C}}(S)\) a division ring, but a commutative division ring is just a field.
\end{proof}

\begin{eg}[\({\FFGMod[R]}\) if and only if \(R\) is a field]
All objects in \(\FFGMod[R]\) are of the form \(R^n\) for some \(n \in \mathbb{N}\).
If \(n=0\) then \(R^n=0\).
If \(n>1\) then \(R^n\) has a subobject \(R \tomono R^n\).
Now suppose \(\FFGMod[R]\) is abelian semisimple.
Then it has at least one abelian simple object, which by elimination can be identified with \(R\).
Since \(\Hom_{\FFGMod[R]}(R) = R\), \cref{cor:abs-obs-implies-field} makes \(R\) a field.

Conversely, if \(R=\field\) is a field then \(\FFGMod[R]=\Vect[\field]\) is abelian semisimple with standard matrix kernels and cokernels.
\label{eg:rmod-abss}
\end{eg}
\subsubsection{\(R\)-linear semisimplicity}
\label{sec:org1a13df0}

We call an \(R\)-linear category \defined{additive} if it has all finite direct sums.
Recall that objects \(X\) and \(Y\) in an \(R\)-linear category are called \defined{disjoint} if we have
\begin{equation*}
\Hom(X,Y) = \Hom(Y,X) = 0 \,.
\end{equation*}

\begin{definition}
An object \(X\) in an \(R\)-linear category is called \defined{\(R\)-linear simple} if it satisfies \(\End(X) \iso R\).
An \(R\)-linear category \(\cat{C}\) is called \defined{\(R\)-linear semisimple} if it is additive and there is a collection \(S\) of mutually disjoint \(R\)-linear simple objects such that each object in \(\cat{C}\) is a finite direct sum of objects in \(S\).
\label{def:objss}
\end{definition}
\begin{eg}[\({\FFGMod[R]}\)]
An \(R\)-module \(M\) is free if and only if \(M \iso \bigoplus_{i \in I} R\) for some set \(I\), in which case \(M\) is also finitely generated if and only if \(I\) is finite.
We also have \(\End(R) \iso R\), so that \(\FFGMod[R]\) is \(R\)-linear semisimple with a single linear simple object \(R\).
\label{eg:objss:modffg}
\end{eg}

\begin{proposition}
Let \(\cat{C}\) be an \(R\)-linear category, and suppose \(\cat{C}\) is abelian semisimple.
If \(\cat{C}\) is hom-finite and \(R\) is an algebraically closed field, then \(\cat{C}\) is also \(R\)-linear semisimple.
\end{proposition}
\begin{proof}
This follows directly from Schur's lemma (\cref{lem:schur}).
\end{proof}

\begin{definition}
The direct sum of additive categories \((\cat{C}_i)_{i \in I}\) is given by the additive category \(\bigboxplus_{i \in I} \cat{C}_i\) whose objects are formal lists  \(\bigboxplus_{i \in I} X_i\) for finitely many non-zero \(X_i \in \cat{C}_i\), and whose morphisms are given by
\begin{equation*}
\Hom_{\bigboxplus_{i \in I} \cat{C}_i}\Big( \bigboxplus_{i \in I} X_i, \bigboxplus_{i \in I} Y_i \Big)
= \bigboxplus_{i \in I} \Hom_{\cat{C}_i}(X_i, Y_i)
\,.
\end{equation*}
\label{def:dirsum-addcat}
\end{definition}

Clearly this direct sum preserves both \(R\)-linear and abelian semisimplicity.
An object \(X \in \bigboxplus_{i \in I} \cat{C}_i\) is (\(R\)-linear or abelian) simple if and only if it has a single non-zero component \(X_i\) which is (\(R\)-linear or abelian) simple in \(\cat{C}_i\).
The below structure theorem is well known, but appears to lack a published proof or explicit statement.

\begin{theorem}
A category \(\cat{C}\) is \(R\)-linear semisimple if and only if there is an \(R\)-linear equivalence
\begin{equation*}
\cat{C} \eqv \bigboxplus_{i \in I} \FFGMod[R] \,.
\end{equation*}
Furthermore, \(I\) is in bijection with the linear simple objects of \(\cat{C}\) up to isomorphism.
\label{thm:objss}
\end{theorem}
\begin{proof}
Suppose \(\cat{C}\) is \(R\)-linear semisimple with simple objects \((S_i)_{i \in I}\) up to isomorphism.
We can view morphisms \(f \in \Hom_{\cat{C}}(X,Y)\) as matrices with components \(f_{ji} \id[S_i] \in \Hom_{\cat{C}}(S_i,S_j)\).
We can then define a fully faithful functor \(\bigboxplus_{i \in I} \FFGMod[R] \to \cat{C}\) acting on objects and morphisms respectively by
\begin{align*}
\bigboxplus_{i \in I} R^{n_i}
&\mapsto
\bigoplus_{i \in I} S_i^{\oplus n_i}
&
\bigboxplus_{i \in I} (f_{ijk})_{j,k \in I}
&\mapsto
\bigboxplus_{i \in I} (f_{ijk} \id[S_i])_{j,k \in I}
\,.
\end{align*}
Essential surjectivity follows directly from the \(R\)-linear semisimplicity of \(\cat{C}\).
The converse follows from the \(R\)-linear semisimplicity of \(\bigboxplus_{i \in I} \FFGMod[R]\).
\end{proof}
\begin{corollary}
Let \(\cat{C}\) be an \(R\)-linear semisimple category.
Then \(\cat{C}\) is abelian semisimple if and only if \(R\) is a field.
\end{corollary}
\begin{proof}
We showed in \cref{eg:rmod-abss} that \(\FFGMod[R]\) is abelian semisimple if and only if \(R\) is a field.
\end{proof}
\subsection{Direct sums and additivity for \texorpdfstring{\(\tau\)}{τ}-graded categories}
\label{sec:additivity}
In order to discuss semisimple \(\tau\)-graded categories, we first need to fix a notion of additivity.
In Section 3.4 of \cite{sv2023}, Sözer and Virelizier define what they call \(H\)-additivity, using \(h\)-direct sums for \(h \in H\).
This generalises immediately to the \(\tau\)-graded case:

\begin{definition}
Let \((X_i)_{i \in I}\) be a collection of objects in a \(\tau\)-graded category, and let \(h \in H\).
An \defined{\(h\)-direct sum} of \((X_i)_{i \in I}\) is an object \(\gbigoplus{h}_{i \in I} X_i\) together with morphisms
\begin{align*}
\pi_i &\in \Hom_{\cat{C}}^{h^{-1}} \Big( \gbigoplus{h}_{i \in I} X_i, X_i \Big)
&
\iota_i &\in \Hom_{\cat{C}}^h \Big( X_i, \gbigoplus{h}_{i \in I} X_i \Big)
\end{align*}
such that \(\sum_{i \in I} \iota_i \of \pi_i = \id[\gbigoplus{h}_{i \in I} X_i]\) and \(\pi_i \of \iota_j = \delta_{ij} \id[X_i]\) for all \(i,j \in I\).
We omit \(h\) when \(h=1\).
\end{definition}
\begin{definition}
A \(\tau\)-graded category is called \defined{\(H\)-additive} if it has all finite \(h\)-direct sums for all \(h \in H\).
\end{definition}

Sözer and Virelizier remark that an \(h\)-direct sum, when it exists, is unique up to a degree 1 isomorphism.
Then a (1-)direct sum in a \(\tau\)-graded category \(\cat{C}\) is the same as an ordinary direct sum in \(\cat{C}^1\).
If \(\cat{C}\) has shifts, then we can always convert a degree \(h\) morphism to a morphism in \(\cat{C}^1\) by composing with an appropriate shift isomorphism.
It is therefore reasonable to expect that many notions can be expressed with reference only to \(\cat{C}^1\).

\begin{definition}
Let \(\cat{C}\) be a \(\tau\)-graded category.
\begin{itemize}
\item We call \(\cat{C}\) \defined{additive} if it has shifts and \(\cat{C}^1\) is additive.
\item We call a \(\tau\)-graded functor or natural transformation additive if its restriction to \(\cat{C}^1\) is additive.
\item We write \(\Cat_{\tau}^{\oplus}\) for the 2-category of additive \(\tau\)-graded categories, additive \(\tau\)-graded functors and additive \(\tau\)-graded natural transformations.
\end{itemize}
We write \((-)^{\oplus}\) for additive completion.
For a \(\tau\)-graded category this means adding both shifts and finite direct sums.
\end{definition}
\begin{remark}
If a \(\tau\)-graded category \(\cat{C}\) already has shifts, then \(\cat{C}^{\oplus}\) has compatible shifts via the degree \(h\) isomorphisms
\begin{equation*}
\begin{tikzcd}[column sep=6em]
\bigoplus_{i=1}^n X_i
\ar[r, "\bigoplus_{i=1}^n r_{X_i,h}"]
&
\bigoplus_{i=1}^n X_i \langle h \rangle
\,.
\end{tikzcd}
\end{equation*}
\end{remark}

The following technical lemma describes precisely how shifts interact with \(h\)-direct sums.
The corollary connects our definition of additivity with that of Sözer and Virelizier.

\begin{lemma}
Let \((X_i)_{i \in I}\) be a collection of objects in a \(\tau\)-graded category, and let \(h \in H\).

\begin{enumerate}
\item Suppose \(\cat{C}\) contains both the 1-direct sum \(\bigoplus_{i \in I} X_i\) and the shift object \((\bigoplus_{i \in I} X_i)\langle h \rangle\).
Then \((\bigoplus_{i \in I} X_i)\langle h \rangle\) is an \(h\)-direct sum of \((X_i)_{i \in I}\).

\item Suppose \(\cat{C}\) contains the shift objects \((X_i\langle h \rangle)_{i \in I}\) and the 1-direct sum \(\bigoplus_{i \in I} X_i\langle h \rangle\).
Then \(\bigoplus_{i \in I} X_i\langle h \rangle\) is an \(h\)-direct sum of \((X_i)_{i \in I}\).

\item Suppose \(\cat{C}\) contains the singleton \(h\)-direct sum \(\gbigoplus{h} X\) for some \(X\).
Then \(\gbigoplus{h} X\) is a shift of \(X\) by \(h\).
\end{enumerate}
\label{lem:gbigoplus}
\end{lemma}
\begin{proof}
Suppose we have \((X,\pi_i,\iota_i) := \bigoplus_{i \in I} X_i\) with shift \((X\langle h \rangle, r_{X,h})\).
We can then define
\begin{align*}
\pi'_i
&:= \pi_i \of r_{X,h}^{-1}
\in \Hom_{\cat{C}}^{h^{-1}}(X \langle h \rangle, X_i)
&
\iota'_i
&:= r_{X,h} \of \iota_i
\in \Hom_{\cat{C}}^h(X_i, X \langle h \rangle)
\end{align*}  
with
\begin{gather*}
\sum_{i \in I} \big( \iota'_i \of \pi'_i \big)
= r_{X,h} \of \sum_{i \in I} (\iota_i \of \pi_i) \of r_{X,h}^{-1}
= r_{X,h} \of \id[X] \of r_{X,h}^{-1}
= \id[X \langle h \rangle]
\\
\pi'_i \of \iota'_j
= (\pi_i \of r_{X,h}^{-1}) \of (r_{X,h} \of \iota_j)
= \pi_i \of \iota_j
= \delta_{ij} \id[X_i]
\end{gather*}  
for all \(i,j \in I\).
The proof for (2) is similar.
Finally, a coprojection \(\iota \colon X \to \gbigoplus{h}_{i=1}^1 X\) is a degree \(h\) isomorphism with inverse given by the corresponding projection.
Therefore we have \(\gbigoplus{h}_{i=1}^1 X = X\langle h \rangle\) with \(r_{X,h} = \iota\).
\end{proof}
\begin{corollary}
A \(\tau\)-graded category is additive if and only if it is \(H\)-additive.
\label{cor:additive-hom-graded-def}
\end{corollary}
\subsection{Semisimplicity for \texorpdfstring{\(\tau\)}{τ}-graded categories}
\label{sec:sstau}
Now that we have a notion of direct sum, and have recalled semisimplicity for \(R\)-linear categories, we are able to give a new definition of semisimplicity for \(\tau\)-graded categories.
We show in \cref{prop:sstau} that this matches the notion used in \cite{sv2023}, combining their Sections 3.1 and 3.5.
We will then define and explore a new parametrised class of semisimple \(\tau\)-graded categories which will later play a role in the structure theorem.
\subsubsection{Definition and basic results}
\label{sec:org277b9ea}

\begin{definition}
Let \(\cat{C}\) be a \(\tau\)-graded category. Then

\begin{itemize}
\item an object in \(\cat{C}\) is \defined{simple} if it is \(R\)-linear simple in \(\cat{C}^1\).

\item \(\cat{C}\) is \defined{\(H\)-semisimple} if
\begin{enumerate}
\item for any \(h \in H\), each object in \(\cat{C}\) can be written as a finite \(h\)-direct sum of simple objects
\item any two simple objects are either isomorphic or disjoint in \(\cat{C}^1\)
\end{enumerate}

\item we call \(\cat{C}\) \defined{semisimple} if \(\cat{C}\) has shifts and \(\cat{C}^1\) is \(R\)-linear semisimple (or equivalently if \(\cat{C}\) is additive and \(\cat{C}^1\) is \(R\)-linear semisimple).
\end{itemize}
\end{definition}

\begin{proposition}
Let \(\cat{C}\) be a \(\field\)-linear \(\tau\)-graded category over a field \(\field\).
Then \(\cat{C}\) is semisimple if and only if it is additive and \(H\)-semisimple.
\label{prop:sstau}
\end{proposition}
\begin{remark}
In one direction this is equivalent to Lemma 3.1(i) of \cite{sv2023}.
\end{remark}
\begin{proof}
First, observe that \cref{lem:gbigoplus} allows us to say that \(\cat{C}\) is additive \(H\)-semisimple if and only if it has shifts and \(\cat{C}^1\) satisfies the following two axioms in addition to additivity:
\begin{enumerate}
\item every object can be written as a finite direct sum of linear simple objects
\item any two simple objects are either isomorphic or disjoint.
\end{enumerate}
This is equivalent to \(\field\)-linear semisimplicity by Schur's lemma (\ref{lem:schur}).
\end{proof}

\begin{eg}[\(R\cat{G}_{\tau}^{\oplus}\)]
Recall the \(\tau\)-graded category \(R\cat{G}_{\tau}\) from \cref{eg:tgrad:grpd}, which has shifts by \cref{eg:shifts:grpd}.
For each \(g,g' \in R\cat{G}_{\tau}\), the additive completion \(R\cat{G}_{\tau}^{\oplus}\) has
\begin{equation*}
\Hom_{R\cat{G}_{\tau}^{\oplus}}^1 (g,g')
= \Hom_{R\cat{G}_{\tau}}^1 (g,g')
= \delta_{\tau(1)g,g'} R (1,g)
\iso \delta_{g,g'} R
\,,
\end{equation*}
so that the inclusion of \(\ob(R\cat{G}_{\tau})\) gives a set of mutually disjoint \(R\)-linear simple objects in \((R\cat{G}_{\tau}^{\oplus})^1\).
Then \(R\cat{G}_{\tau}^{\oplus}\) is semisimple by construction.
The equivalence classes of simple objects up to general isomorphism are in bijection with the coset space \(\im(\tau) \backslash G\).
\label{eg:ss:grpd}
\end{eg}

\begin{eg}[\({\FFGMod[R]_{\tau}}\)]
Recall the \(\tau\)-graded category \(\Mod[R]_{\tau}\) from \cref{eg:tgrad:modtau}, which has shifts by \cref{eg:shifts:modtau}.
Define \(\FFGMod[R]_{\tau}\) to be the \(\tau\)-graded subcategory containing only the \(G\)-graded modules which are free and finitely generated.
Then we have
\begin{equation*}
(\FFGMod[R]_{\tau})^1
= \FFGMod[R]_G
\eqv \bigboxplus_{g \in G} \FFGMod[R]
\end{equation*}
as \(R\)-linear categories, so that \(\FFGMod[R]_{\tau}\) is semisimple by \cref{eg:objss:modffg}.
The simple objects are given up to degree 1 isomorphism by symbols \((R_g)_{g \in G}\) with components \((R_g)_{g'} = \delta_{gg'} R\).
Up to general isomorphism they correspond to representatives of the coset space \(\im(\tau)\backslash G\).
\label{eg:ss:modtau}
\end{eg}

\begin{eg}[\(\tau \colon C_8 \to C_2\)]
For \(H=C_8=\langle x \rangle\) and \(G=C_2=\langle y \rangle\), recall the \(\tau\)-graded categories \((\cat{C}_k)_{k=1,2,4,8}\) from \cref{eg:tgrad:cyclic}, which have shifts by \cref{eg:shifts:cyclic}.
We already have each \(\End_{\cat{C}_k}^1 = R\), so the additive completion \(\cat{C}_k^{\oplus}\) is semisimple.
The \(8/k\) simple objects are distinct up to degree 1 isomorphism, but the whole category is connected up to general isomorphism.
\label{eg:ss:cyclic}
\end{eg}
\subsubsection{The \texorpdfstring{\(\tau\)}{τ}-graded category \texorpdfstring{\(\cat{M}_{\tau}(L,\psi)\langle g \rangle\)}{Mτ(L,ψ)<g>}}
\label{sec:orgc6908c4}

Our final class of examples is more involved, but will play a central role in our structure theorem for semisimple \(\tau\)-graded categories (\cref{sec:struct}).
They turn out to be indecomposable, and a consequence of \cref{thm:tgrad:struct} will be that all indecomposable semisimple \(\tau\)-graded categories are of this form.

Given a subgroup \(L \le H\), recall that \(M := \Hom_{\Set}(H/L,R^{\times})\) is a right \(H\)-module with multiplication in \(R^{\times}\) and right \(H\)-action \((f \triangleleft h)(kL) = f(hkL)\).
Write \(1 \colon H/L \to R^{\times}\) for the constant function \(hL \mapsto 1\).
A \defined{normalised 2-cocycle} for \(M\) is then a function \(\psi \colon H^2 \to M\) with
\begin{equation*}
d\psi(a,b,c)
= \psi(b,c) \cdot \psi(ab,c)^{-1} \cdot \psi(a,bc) \cdot \psi(a,b)^{-1} \triangleleft c
= 1
\end{equation*}
and all \(\psi(1,h) = \psi(h,1) = 1\).

\begin{definition}
Let \(L \le \ker(\tau)\ \triangleleft H\) and let \(\psi \colon H^2 \to \Hom_{\Set}(H/L,R^{\times})\) be a normalised 2-cocycle.
Write \(\cat{S}_{\tau}(L,\psi)\langle g \rangle\) for the \(\tau\)-graded category with
\begin{itemize}
\item objects \((R_{hL})_{hL \in H/L}\)
\item degree \(a\) morphisms given by \(\Hom^a(R_{hL},R_{h'L}) = \delta_{ahL,h'L} R e_{hL}^a\) with each \(e_{hL}^1 = \id[R_{hL}]\)
\item composition given by \(e_{bhL}^a \of e_{hL}^b = \psi(a,b)(hL)^{-1} e_{hL}^{ab}\)
\item object degrees given by \(|R_{hL}| = \tau(h)g\).
\end{itemize}
We then write \(\cat{M}_{\tau}(L,\psi)\langle g \rangle\) for the completion of \(\cat{S}_{\tau}(L,\psi)\langle g \rangle\) under direct sums, and we also write \(\cat{M}_{\tau}(L,\psi) = \cat{M}_{\tau}(L,\psi)\langle 1 \rangle\).
\label{def:mtau}
\end{definition}
\begin{remark}
Restricting \(L\) to a subgroup of \(\ker(\tau)\) is necessary in order to ensure a \(\tau\)-graded category.
This is because \cref{def:tgrad} requires \(|R_{h'L}| = \tau(a)|R_{hL}|\) whenever \(\Hom^a(R_{hL},R_{h'L}) \neq 0\), but we have \(\Hom^a(R_{hL},R_{h'L}) \neq 0\) precisely when there exists \(\ell \in L\) with \(ah = h'\ell\).
Then we have
\begin{equation*}
\tau(h')g
= |R_{h'L}|
= \tau(a)|R_{hL}|
= \tau(ah)g
= \tau(h'\ell)g
\end{equation*}
if and only if \(\ell \in \ker(\ell)\).
\end{remark}

\begin{eg}[\(\tau \colon C_8 \to C_2\)]
For \(H=C_8=\langle x \rangle\) and \(G=C_2=\langle y \rangle\), recall the \(\tau\)-graded categories \((\cat{C}_k)_{k=1,2,4,8}\) from \cref{eg:tgrad:cyclic}, which have shifts by \cref{eg:shifts:cyclic}.
Directly from \cref{eg:ss:cyclic}, we have each \(\cat{C}_k \iso \cat{M}_{\tau}(C_k,1)\).
\end{eg}

\begin{lemma}
Each \(e_{hL}^a\) is invertible with \((e_{hL}^a)^{-1} = \psi(a^{-1},a)(hL) e_{ahL}^{a^{-1}}\) and \(e_{hL}^1 = \id[R_{hL}]\).
\label{lem:mtau:basis-morphisms}
\end{lemma}
\begin{proof}
For any \(h,a \in H\), the normalisation of \(\psi\) gives
\begin{align*}
e_{ahL}^1 \of e_{hL}^a &= \psi(1,a)(hL) e_{hL}^a = e_{hL}^a &
e_{hL}^a  \of e_{hL}^1 &= \psi(a,1)(hL) e_{hL}^a = e_{hL}^a \,.
\end{align*}
It follows that
\begin{gather*}
\begin{split}
e_{hL}^a \of \psi(a^{-1},a)(hL) e_{ahL}^{a^{-1}}
&= \psi(a,a^{-1})(ahL)^{-1} \cdot \psi(a^{-1},a)(hL) \cdot \id[R_{ahL}] \\
&= (\psi(a,a^{-1}) \triangleleft a \big)(hL)^{-1} \cdot \psi(a^{-1},a)(hL) \cdot \id[R_{ahL}] \\
&= \psi(1,a)(hL) \cdot \psi(a,1)(hL)^{-1} \cdot \id[R_{ahL}] \\
&= \id[R_{ahL}]
\end{split}
\\
\psi(a^{-1},a)(hL) e_{ahL}^{a^{-1}} \of e_{hL}^a
= \psi(a^{-1},a)(hL)^{-1} \cdot \psi(a^{-1},a)(hL) \id[R_{hL}]
= \id[R_{hL}]
\,.\qedhere
\end{gather*}
\end{proof}

\begin{proposition}
Every \(\tau\)-graded category of the form \(\cat{M}_{\tau}(L,\psi)\langle g \rangle\) is indecomposable semisimple, with shifts \(r_{R_{hL},h} := e_{hL}^a \colon R_{hL} \to R_{ahL}\) and a single simple object \(R_{1L}\) up to general isomorphism.
\label{prop:indecss}
\end{proposition}
\begin{proof}
By definition, \(\cat{M}_{\tau}(L,\psi)\langle g \rangle^1\) is \(R\)-linear semisimple with simple objects \((R_{hL})_{hL \in H/L}\) which are all connected by the general morphisms \(e_{hL}^a\). These are invertible by \cref{lem:mtau:basis-morphisms}, and hence provide shifts.

We note that a semisimple \(\tau\)-graded category is indecomposable if and only if it has a single simple object up to general degree isomorphism: if \(\cat{C} \iso \cat{A} \boxplus \cat{B}\) for semisimple \(\cat{A}\) and \(\cat{B}\) with simple \(A \in \cat{A}\), \(B \in \cat{B}\), then \(A \boxplus 0\) and \(0 \boxplus B\) are two non-isomorphic simple objects in \(\cat{C}\).
\end{proof}

\begin{proposition}
There exist degree \(x\) simple objects in \(\cat{M}_{\tau}(L,\psi)\langle g \rangle\) if and only if \(x \in \im(\tau)g\), in which case they are in bijection with \(\ker(\tau)/L\).
\label{prop:mtau-gsimples}
\end{proposition}
\begin{proof}
Any degree \(x\) simple object is of the form \(R_{hL}\) with \(x = \tau(h)g\).
Conversely, suppose \(x = \tau(b)g \in \im(\tau)g\).
Then for each \(hL \in H/L\) we have \(|R_{hL}| = \tau(h)g = x\) if and only if \(hL = kbL\) for some \(k \in \ker(\tau)\).
\end{proof}

We note that two simple objects in \(\cat{M}_{\tau}(L,\psi)\langle g \rangle\) are equal if and only if they are degree 1 isomorphic.
In a sense this is the highest level of skeletality achievable in a semisimple \(\tau\)-graded category.
\subsection{Structure theorem}
\label{sec:struct}
Finally we prove the structure theorem for semisimple \(\tau\)-graded categories.
This gives a direct sum decomposition of any such category, where each summand is of the form \(\cat{M}_{\tau}(L,\psi)\langle g \rangle\) for some \(L,\psi,g\).
Our first result in this section is a technical lemma that describes how to obtain values for \(L\) and \(\psi\).
\Cref{thm:tgrad:struct} is the main result.
We follow this and end the section with \cref{thm:mtau-equiv}, which determines precisely when and how two categories of the form \(\cat{M}_{\tau}(L,\psi)\langle g \rangle\) can be \(\tau\)-graded equivalent.

\begin{lemma}
Let \(\cat{C}\) be a \(\tau\)-graded category with shifts and a simple object \(S\).
Write
\begin{equation*}
L_S = \{ a \in H \st S\langle a \rangle \iso S \} \le \ker(\tau)
\end{equation*}
and fix coset representatives \((h_i)_{i \in I}\) so that \(H = \coprod_{i \in I} h_i L_S\).
It follows that \(I\) has an \(H\)-action \(i \mapsto ai\) induced by \(h_{ai}L_S = ah_i L_S\).

For all \(i \in I\) and \(a \in H\), fix a morphism \(f_{S,i}^a\) spanning \(\Hom_{\cat{C}}^a(S\langle h_i \rangle, S\langle h_{ai} \rangle) \iso R\).
In particular, fix each \(f_{S,i}^1 = \id[S\langle h_i \rangle]\).
Then
\begin{equation}
\label{eq:tgrad:psi-s}
f_{S,i}^{ab} = \psi_S(a,b)(h_iL_S) \cdot f_{S,bi}^a \of f_{S,i}^b \,.
\end{equation}
defines a normalised 2-cocycle \(\psi_S \colon H^2 \to \Hom_{\Set}(H/L_S,R^{\times})\).
\label{lem:tgrad:psi-s}
\end{lemma}
\begin{proof}
Since \(f_{S,bi}^a \of f_{S,i}^b\) is invertible, it gives a non-zero morphism in \(\Hom_{\cat{C}}^{ab}(S\langle h_i \rangle, S\langle h_{abi} \rangle) = R f_{S,i}^{ab}\).
Since each \(i\) picks out a unique coset of \(L_S\), we must have some function \(\psi_L \colon H^2 \to \Hom_{\Set}(H/L_S,R^{\times})\) satisfying \eqref{eq:tgrad:psi-s}.
We then have
\begin{align*}
\begin{split}
f_{S,bci}^{a} \of (f_{S,ci}^b \of f_{S,i}^c)
&= \psi_S(b,c)(h_i L_S)^{-1} \cdot f_{S,bci}^a \of f_{S,i}^{bc} \\
&= \psi_S(b,c)(h_i L_S)^{-1} \cdot \psi_S(a,bc)(h_i L_S)^{-1} \cdot f_{S,i}^{abc}
\end{split}
\\
\begin{split}
(f_{S,bci}^a \of f_{S,ci}^b) \of f_{S,i}^c
&= \psi_S(a,b)(h_{ci} L_S)^{-1} \cdot f_{S,ci}^{ab} \of f_{S,i}^c \\
&= \psi_S(a,b)(ch_i L_S)^{-1} \cdot \psi_S(ab,c)(h_i L_S)^{-1} \cdot f_{S,i}^{abc} \,,
\end{split}
\end{align*}
giving the 2-cocycle condition by the associativity of composition in \(\cat{C}\).
Normalisation follows from \(f_{S,i}^1 = \id[S\langle h_i \rangle]\) by
\begin{equation*}
\id[S\langle h_{ai} \rangle] \of f_{S,i}^a
= \psi_S(1,a)(h_iL_S) \cdot f_{S,i}^a
\qquad
f_{S,i}^a \of \id[S\langle h_i \rangle]
= \psi_S(a,1)(h_iL_S) \cdot f_{S,i}^a
\,.\qedhere
\end{equation*}
\end{proof}

\begin{theorem}
A \(\tau\)-graded category \(\cat{C}\) is semisimple if and only if there is a \(\tau\)-graded equivalence
\begin{equation*}
\cat{C} \eqv \bigboxplus_{i \in I} \cat{M}_{\tau}(L_i,\psi_i)\langle g_i \rangle
\end{equation*}
for some subgroups \(L_i \le \ker(\tau)\), 2-cochains \(\psi_i \colon H^2 \to \Hom_{\Set}(H/L_i,R^{\times})\) and elements \(g_i \in G\).
Suppose \(\cat{C}\) is semisimple, and fix a representative collection \((S_i)_{i \in I}\) of simple objects in \(\cat{C}\) up to general isomorphism.
Then we can choose each \(g_i = |S_i|\), \(L_i = L_{S_i}\) and \(\psi_i = \psi_{S_i}\) as in \cref{lem:tgrad:psi-s}.
\label{thm:tgrad:struct}
\end{theorem}
\begin{proof}
We can assume \(\cat{C}\) is additive without loss of generality.
Fix a representative collection \((S_i)_{i \in I}\) of simple objects in \(\cat{C}\) up to general isomorphism, even if \(\cat{C}\) is not semisimple.
Set each \(g_i = |S_i|\), \(L_i = L_{S_i}\) and \(\psi_i = \psi_{S_i}\) as in \cref{lem:tgrad:psi-s}.
Fix also coset representatives \((h_{ij})_{j \in J_i}\) for each \(L_i\) so that \(H = \coprod_{j \in J_i} h_{ij}L_i\).

We can then define a \(\tau\)-graded functor \(F \colon \bigboxplus_{i \in I} \cat{M}_{\tau}(L_i,\psi_i)\langle g_i \rangle \to \cat{C}\) componentwise by
\begin{equation*}
R_{h_{ij}L_i} \mapsto S_i\langle h_{ij} \rangle
\qquad
e_{h_{ij}L_i}^a \mapsto f_{S_i,j}^a
\,,
\end{equation*}
using the morphisms \(f_{S_i,j}^a\) also from \cref{lem:tgrad:psi-s}.
Using coset representatives ensures that this is well-defined and maps non-isomorphic objects to non-isomorphic objects.
Clearly \(F\) preserves object and morphism degrees, so \(\tau\)-graded functoriality follows from the composition rule in \cref{lem:tgrad:psi-s} .
Moreover, \(F\) is fully faithful because it maps non-zero morphisms to non-zero morphisms and every \(a\)-degree homset in the source category is at most one-dimensional.
Therefore \(F\) picks out the largest semisimple \(\tau\)-graded subcategory of \(\cat{C}\), which is \(\cat{C}\) if and only if \(\cat{C}\) is semisimple.
\end{proof}
\begin{corollary}
An \(H\)-Hom-graded category is semisimple if and only if it exhibits an \(H\)-Hom-graded equivalence with some \(\bigboxplus_{i \in I} \cat{M}_{(H \toepi 1)}(L_i,\psi_i)\).
\end{corollary}

\begin{eg}[\(R\cat{G}_{\tau}^{\oplus} \iso {\FFGMod[R]_{\tau}}\)]
Recall from \cref{eg:ss:grpd} that \(R\cat{G}_{\tau}^{\oplus}\) is semisimple with representative simple objects given by representatives \((g_i)_{i \in I}\) of the coset space \(\im(\tau) \backslash G\).
Each \(g_i\) also has degree \(g_i \in G\), and we have \(g_i = g_i \langle a \rangle = \tau(a)g_i\) if and only if \(a \in \ker(\tau)\)
Composition is induced directly by group multiplication.

Moreover, recall from \cref{eg:ss:modtau} that \(\FFGMod[R]_{\tau}\) also has simple objects indexed up to general isomorphism by the coset space \(\im(\tau)\backslash G\).
Each simple object \(R_{g_i}\) has degree \(g_i\), and we again have \(R_{g_i} = R_{g_i}\langle a \rangle = R_{\tau(a)g_i}\) if and only if \(a \in \ker(\tau)\).
We have
\begin{equation*}
r_{R_{\tau(a)g},a'} \of r_{R_g,a} = r_{R_g,a'a}
\end{equation*}
because the composition of permutations is always another permutation.
Therefore we have \(\tau\)-graded equivalences
\begin{equation*}
R\cat{G}_{\tau}^{\oplus}
\eqv \FFGMod[R]_{\tau}
\eqv \bigboxplus_{\im(\tau)g \in \im(\tau) \backslash G} \cat{M}_{\tau}(\ker(\tau),1)\langle g \rangle
\,.
\end{equation*}
As a special case, \(\tau = \id[H]\) gives \(\cat{M}_{\id[H]}(1,1) \eqv (\FFGMod[R]_H)^{\bullet}\).
\end{eg}

\begin{theorem}
Each \(\tau\)-graded equivalence
\begin{equation*}
\cat{M}_{\tau}(L,\psi)\langle g \rangle \to \cat{M}_{\tau}(L',\psi')\langle g' \rangle
\end{equation*}
corresponds to a choice of group element and normalised 1-cochain
\begin{equation*}
t \in \tau^{-1}(gg'^{-1})
\qquad
\gamma \colon H \to \Hom_{\Set}(H/L,R^{\times})
\end{equation*}
with \(L = tL't^{-1}\) and \(\psi = \psi'^t \cdot d\gamma\), writing \(\psi'^{t} (a,b) (hL') = \psi' (a,b) (thL')\).
The equivalence induced by such data is given by
\begin{equation*}
F_{t,\gamma} \colon
R_{hL} \mapsto R'_{htL'}
\,,\qquad
e_{hL}^a \mapsto \gamma(a)(hL) {e'}_{htL'}^a
\,,
\end{equation*}
using primes to denote the simple objects and basis morphisms in \(\cat{M}_{\tau}(L',\psi')\langle g' \rangle\).
Natural isomorphisms \(F_{t,\gamma} \tonat F_{s,\delta}\) require \(tL' = sL'\) and correspond to 0-cochains \(\eta \colon H/L \to R^{\times}\) with \(\gamma = \delta \cdot d\eta\).
\label{thm:mtau-equiv}
\end{theorem}
\begin{proof}
Write \(\cat{M} = \cat{M}_{\tau}(L,\psi)\langle g \rangle\) and \(\cat{M}' = \cat{M}_{\tau}(L',\psi')\langle g' \rangle\).
Since \(\cat{M}\) has a single simple object \(R_{1L}\) up to general isomorphism, any functor \(F \colon \cat{M} \to \cat{M}\) is determined on objects by the image of \(R_{1L}\) and on morphisms by the image of \(e_{hL}^a\) for each \(h,a \in H\).

If \(F\) is a \(\tau\)-graded equivalence, it must preserve simplicity and object degrees.
Since \(|R_{1L}| = g\), we must therefore have \(FR_{1L} = R'_{tL'}\) for some \(t \in H\) with \(|R'_{tL'}| = \tau(t)g' = g\); that is, \(t \in \tau^{-1}(gg'^{-1})\).
More generally, \(F\) must map each simple object \(R_{hL}\) to some \(R'_{aL'}\).
Then we have
\begin{equation*}
\delta_{htL', aL'} R {e'}_{tL'}^h
= \Hom_{\cat{M}}^h (R'_{tL'}, R'_{aL'})
\iso \Hom_{\cat{M}}^h (R_{1L}, R_{hL})
= R e_{1L}^h
\end{equation*}
so that \(aL' = htL'\).
Therefore \(F\) is given on objects by \(R_{hL} \mapsto R'_{htL'}\).
Well-definedness requires \(R'_{atL'} = R'_{btL'}\) whenever \(R_{aL} = R_{bL}\).
Writing \(\ell = a^{-1}b\), this is the same as requiring \(t^{-1} \ell t = (at)^{-1} bt \in L'\) whenever \(\ell \in L\); that is, \(L \le t L' t^{-1}\).

On morphisms, \(F\) must now act by \(e_{hL}^a \mapsto \gamma(a)(hL) \cdot {e'}_{htL'}^a\) for some function
\begin{equation*}
\gamma \colon H \to \Hom_{\Set}(H/L,R^{\times}) \,.
\end{equation*}
Functoriality requires \(\gamma(1)(hL) = 1\) and
\begin{equation*}
\begin{split}
\psi(a,b)(hL)^{-1} \cdot \gamma(ab)(hL) \cdot {e'}_{htL'}^{ab}
&= F\big( \psi(a,b)(hL)^{-1} \cdot e_{hL}^{ab} \big) \\
&= F(e_{bhL}^a \of e_{hL}^b) \\
&= Fe_{bhL}^a \of Fe_{hL}^b \\
&= \gamma(a)(bhL) \cdot {e'}_{bhtL'}^a \of \gamma(b)(hL) \cdot {e'}_{htL'}^b \\
&= \gamma(a)(bhL) \cdot \gamma(b)(hL)
\cdot {\psi'}(a,b)(htL')^{-1} \cdot {e'}_{htL'}^{ab} \,,
\end{split}
\end{equation*}
which rearranges to \(\psi = {\psi'}^t \cdot d\gamma\).
Then \(F\) is
\begin{itemize}
\item full automatically, because the image of each \(e_{hL}^a\) is non-zero and all simple objects in \(\cat{M}'\) have one-dimensional endomorphism spaces
\item surjective automatically, because any \(R'_{hL'}\) has \(R'_{hL'} = F_{t,\gamma}(R_{ht^{-1}L})\)
\item faithful if and only if \(atL' = btL'\) implies \(hL = L\); that is, if and only if \(tL't^{-1} \le L\).
\end{itemize}

A \(\tau\)-graded natural isomorphism \(F_{t,\gamma} \tonat F_{s,\delta}\) has degree 1 components of the form
\begin{equation*}
\eta_{R_{hL}}
= \eta(hL) \id[R'_{htL'}]
\colon R'_{htL'} \to R'_{hsL'}
\end{equation*}
for some function \(\eta \colon H/L \to R^{\times}\).
This is well-defined if and only if \(tL' = sL'\), and the naturality square
\begin{equation*}
\begin{tikzcd}[column sep=5em]
F_{s,\gamma} R_{hL} = R'_{hsL'}
\ar[r,"\eta(hL) \id"]
\ar[d,swap,"\gamma(a)(hL) {e'}_{hL'}^a"]
&
F_{s,\delta} R_{hL} = R'_{hsL'}
\ar[d,"\delta(a)(hL) {e'}_{hL'}^a"]
\\
F_{t,\gamma} R_{ahL} = R'_{ahtL'}
\ar[r,"\eta(ahL) \id"]
&
F_{s,\delta} R_{khL} = R'_{ahsL'}
\,,
\end{tikzcd}
\end{equation*}
commutes if and only if \(\gamma = \delta \cdot d\eta\).
\end{proof}

The use (here with \(n=2\)) of \(n\)-cocycles for classifying objects, \((n-1)\)-cochains for classifying 1-morphisms and \((n-2)\)-cochains for classifying 2-morphism can seem reminiscent of Theorem 8.3.7 in \cite{bl2004}.
Note however that Baez and Lauda have instead \(n=3\).
Baez and Lauda are also dealing with 2-groups (which we have generalised to group homomorphisms; see \cref{sec:defs}), but the key difference is that they vary the 2-group whereas we fix the group homomorphism and vary other data.
\section{Relationship with module categories}
\label{sec:2eqv}
In this section we construct an alternate picture of \(\tau\)-graded categories in terms of module categories.
We first recall the notion of \(H\)-module category from \cite{egno2015} and describe their 2-category \(\ModCat[H]\).
In the special case \(G=1\), we then show that shifts on an \(H\)-Hom-graded category \(\cat{C}\) induce an \(H\)-module category structure on \(\cat{C}^1\).
We extend this to a 2-functor
\begin{equation*}
\Cat^{H,\shifts} \to \ModCat[H] \,,
\end{equation*}
where \(\Cat^{H,\shifts}\) is the 2-category of \(H\)-Hom-graded categories with shifts (\cref{def:hasshifts}).
We construct a 2-functor
\begin{equation*}
\ModCat[H] \to \Cat^{H,\shifts}
\end{equation*}
in \cref{sec:2fun2}, and then show these form a 2-equivalence in \cref{sec:2eqv1}.
Finally, in \cref{sec:2eqv2} we define the 2-category \(\ModCat[H]_{\tau}\) and derive the general 2-equivalence \(\Cat_{\tau}^{\shifts} \eqv \ModCat[H]_{\tau}\).
\subsection{Review of \(H\)-module categories}
\label{sec:org269d62a}

The data used in \cref{thm:tgrad:struct} is very similar to that used in the structure theorem for \(H\)-module categories.
The indecomposable \(H\)-module categories are described in \cite[, Example 7.4.10]{egno2015}, with more complete treatments in \cite{ostrik2003} and \cite{natale2017}.
Following \cite[, \S 2.7]{egno2015} we define \(H\)-module categories as an alias for \(R\)-linear \(\Disc(H)\)-module categories.

\begin{definition}
An \defined{\(H\)-module category} \((\cat{C},\alpha)\) is an \(R\)-linear category \(\cat{C}\) together with an \(R\)-linear monoidal functor \(\alpha^{(-)} \colon \Disc(H) \to \Aut(\cat{C})\) called the \defined{\(H\)-action}.
Explicitly, \(\alpha\) is a collection \((\alpha^h)_{h \in H}\) of \(R\)-linear equivalences \(\cat{C} \to \cat{C}\) together with natural isomorphisms \(\epsilon \colon \id[\cat{C}] \tonat \alpha^1\) and \(\mu_{a,b} \colon \alpha^a \alpha^b \tonat \alpha^ab\) making the diagrams
\begin{equation*}
\begin{tikzcd}
\alpha^h
\ar[r,"\epsilon \hof \id"] \ar[d,swap,"\id \hof \epsilon"] \ar[dr,"\id" description]
&
\alpha^1 \alpha^h
\ar[d,"\mu_{1,h}"]
\\
\alpha^h \alpha^1
\ar[r,swap,"\mu_{h,1}"]
&
\alpha^h
\end{tikzcd}
\qquad
\begin{tikzcd}[column sep=4em]
\alpha^a \alpha^b \alpha^c
\ar[r,"\mu_{a,b} \hof \id"] \ar[d,swap,"\id \hof \mu_{bc}"]
&
\alpha^{ab} \alpha^c
\ar[d,"\mu_{ab,c}"]
\\
\alpha^a \alpha^{bc}
\ar[r,swap,"\mu_{a,bc}"]
&
\alpha^{abc}
\end{tikzcd}
\end{equation*}
commute.
When ambiguous, we write \(\epsilon^{\alpha}\) and \(\mu^{\alpha}\).
\end{definition}
\begin{definition}
An \defined{\(H\)-module functor} \((F,s) \colon (\cat{C},\alpha) \to (\cat{D},\beta)\) is an \(R\)-linear functor \(F \colon \cat{C} \to \cat{D}\) together with natural isomorphisms \((s^h \colon \beta^h F \tonat F \alpha^h)_{h \in H}\) making the diagrams
\begin{equation*}
\begin{tikzcd}
&
F
\ar[dl,swap,"\epsilon^{\beta} \hof \id"] \ar[dr,"\id \hof \epsilon^{\alpha}"]
&
\\
\beta^1 F
\ar[rr,swap,"s^1"]
&&
F \alpha^1
\end{tikzcd}
\qquad
\begin{tikzcd}
\beta^a \beta^b F
\ar[r,"\id \hof s^b"] \ar[d,swap,"\mu^{\beta}_{a,b} \hof F"]
&
\beta^a F \alpha^b
\ar[r,"s^a \hof \id"]
&
F \alpha^a \alpha^b
\ar[d,"\id \hof \mu^{\alpha}_{a,b}"]
\\
\beta^{ab} F
\ar[rr,swap,"s^{ab}"]
&&
F \alpha^{ab}
\end{tikzcd}
\end{equation*}
commute.
An \defined{\(H\)-module natural transformation} \((E,r) \tonat (F,s)\) is an \(R\)-linear natural transformation \(\eta \colon E \tonat F\) with commuting
\begin{equation*}
\begin{tikzcd}
\beta^h E
\ar[r,"r^h"] \ar[d,swap,"\beta^h \hof \eta"]
&
E \alpha^h
\ar[d,"\eta \hof \alpha^h"]
\\
\beta^h F
\ar[r,swap,"s^h"]
&
F \alpha^h
\end{tikzcd}
\end{equation*}
for all \(h \in H\).
Compatible \(H\)-module functors compose as \((E,r) \of (F,s) = (EF,rs)\) for
\begin{equation*}
(rs)^h_X = E s^h_X \of r^h_{FX} \,.
\end{equation*}
We write \(\ModCat[H]\) for the 2-category of \(H\)-module categories, \(H\)-module functors and \(H\)-module natural transformations.
\end{definition}
\subsection{Shifts as group actions; the 2-functor \texorpdfstring{\((-)^1 \colon \Cat^{H,\shifts} \to \ModCat[H]\)}{(-)1: CatH}}
\label{sec:orgd9d46cb}

We show now how shifts objects on a \(\tau\)-graded category \(\cat{C}\) correspond directly to an action of \(H\) on \(\cat{C}^1\).
Recall the definitions and notation from \cref{sec:shifts}.

\begin{definition}
Let \(\cat{C}\) be an \(H\)-Hom-graded category.
A \defined{shift functor} for \(\cat{C}\) is a functor
\begin{equation*}
\phi \colon \Disc(H) \to \Aut(\cat{C}^1)
\end{equation*}
(that is, an autoequivalence \(\phi^a \colon \cat{C}^1 \to \cat{C}^1\) for each \(a \in H\)) together with a fixed degree \(a\) isomorphism \(r_{X,a}^{\phi} \colon X \to \phi^a X\) for each \(X \in \cat{C}\) and \(a \in H\) (making each \(\phi^a X\) a shift of \(X\) by \(a\)).

We leave such isomorphisms unlabelled when they appear in commutative diagrams; their labels are clear from the source and target.
Furthermore, unlabeled arrows of the form \(\phi^a X \to \phi^a \phi^b X\) should be read as \(r^{\phi}_{\phi^b X} \of r^{\phi}_{X,b} \of (r^{\phi}_{X,a})^{-1}\) so that the square
\begin{equation*}
\begin{tikzcd}[row sep=.7em]
\phi^a X \ar[r]
&
\phi^a \phi^b X
\\
X \ar[r] \ar[u]
&
\phi^b X \ar[u]
\end{tikzcd}
\end{equation*}
commutes.
\end{definition}

\begin{lemma}
Let \(\cat{C}\) be an \(H\)-Hom-graded category.
Then \(\cat{C}\) has shifts (in the sense of \cref{def:hasshifts}) if and only if it has a shift functor.

In particular, each fixed choice of shifts \((r_{X,a} \colon X \to X\langle a \rangle)_{X \in \cat{C},a \in H}\) defines a shift functor \(\phi\) mapping objects \(X \mapsto X\langle a \rangle\) and morphisms \(f \in \Hom_{\cat{C}^1}(X,Y)\) to \(r_{Y,a} \of f \of r_{X,a}^{-1}\), with each \(r_{X,a} = r_{X,a}^{\phi}\) .
\end{lemma}
\begin{proof}
Clearly a shift functor directly gives a choice of shifts.
Conversely, suppose \(\cat{C}\) has fixed shifts and define \(\phi\) as above.
For each morphism \(f\) in \(\cat{C}^1\) we have \(|\phi^a f| = a1a^{-1} = 1\), and functoriality of \(\phi^a\) is given by
\begin{equation*}
\phi^a \id[X]
= r_{X,a} \of \id[X] \of r_{X,a}^{-1}
= r_{X,a} \of r_{X,a}^{-1}
= \id[X\langle a \rangle]
= \id[\phi^a X]
\end{equation*}
and the commuting diagram
\begin{equation*}
\begin{tikzcd}[row sep=.7em]
\phi^a X \ar[rr,"\phi^a (f' \of f)"]
&&
\phi^a Z
\\
X \ar[r,"f'"]
\ar[u] \ar[d]
&
Y \ar[r,"f"]
\ar[d]
&
Z
\ar[u] \ar[d]
\\
\phi^a X \ar[r,swap,"\phi^a f'"]
&
\phi^a Y \ar[r,swap,"\phi^a f"]
&
\phi^a Z
\,.
\end{tikzcd}
\end{equation*}
This trivially extends to a functor \(\phi \colon \Disc(H) \to \Aut(\cat{C}^1)\) by the discreteness of \(\Disc(H)\), and it it is straightforward to show that we have \(\phi^{a^{-1}} \phi^a \iso \id[\cat{C}^1] \iso \phi^a \phi^{a^{-1}}\).
\end{proof}

\begin{proposition}
Shift functors are unique up to natural isomorphism.
Let \(\cat{C}\) be an \(H\)-Hom-graded category with shift functors \(\phi\) and \(\psi\).
Then there is a natural isomorphism \(\phi \iso \psi\).
\end{proposition}
\begin{proof}
Suppose \(\cat{C}\) is an \(H\)-Hom-graded category with shift functors \(\phi\) and \(\psi\).
We can define isomorphisms \(\eta^a_X \in \Hom_{\cat{C}^1}(\psi^a X, \phi^a X)\) by the composition
\begin{equation*}
\begin{tikzcd}
\psi^a X
\ar[rr,"\eta^a_X",rounded corners,to path={
  (\tikztostart.north east)
  -- ([yshift=.2em,xshift=1em]\tikztostart.north east)
  -- ([yshift=.2em,xshift=-1em]\tikztotarget.north west) \tikztonodes
  -- (\tikztotarget.north west)
}]
&
X \ar[l] \ar[r]
&
\phi^a X
\,.
\end{tikzcd}
\end{equation*}
Naturality in \(a\) and \(X\) is by discreteness of \(\Disc(H)\) and the commuting diagram
\begin{equation*}
\begin{tikzcd}[row sep=0]
\psi^a X \ar[rrr,"\psi^a f"] \ar[dd,swap,"\eta^a_X"]
&&&
\psi^a Y \ar[dd,"\eta^a_Y"]
\\
&
X \ar[r,"f"]
\ar[ul] \ar[dl]
&
Y
\ar[ur] \ar[dr]
\\
\phi^a X \ar[rrr,swap,"\phi^a f"]
&&&
\phi^a Y
\,.
\end{tikzcd}
\end{equation*}
\end{proof}

\begin{proposition}
Let \(\cat{C}\) be an \(H\)-Hom-graded category with fixed shifts inducing a shift functor \(\phi\).
Then \((\cat{C}^1,\phi)\) defines an \(H\)-module category with \(\epsilon^{\phi}_X = r_{X,1}\) and \(\mu^{\phi}\) given by the composition
\begin{equation*}
\begin{tikzcd}
\phi^a \phi^b X
\ar[rrr,"(\mu^{\phi}_{a,b})_X",rounded corners,to path={
  (\tikztostart.north east)
  -- ([yshift=.2em,xshift=1em]\tikztostart.north east)
  -- ([yshift=.2em,xshift=-1em]\tikztotarget.north west) \tikztonodes
  -- (\tikztotarget.north west)
}]
& \phi^b X \ar[l]
& X \ar[l] \ar[r]
& \phi^{ab} X
\,.
\end{tikzcd}
\end{equation*}
\label{prop:shiftaction}
\end{proposition}
\begin{proof}
Naturality and coherence for \(\mu^{\phi}\) are by the commuting diagrams
\begin{equation*}
\begin{tikzcd}
\phi^a \phi^b X
\ar[rrr,"(\mu^{\phi}_{a,b})_X"]
\ar[ddd,swap,"\phi^a \phi^b f" description]
&[-2em]&[-1em]&[-2em]
\phi^{ab} X
\ar[ddd,"\phi^{ab} f" description]
\\[-1.5em]
& \phi^b X
\ar[ul] \ar[d,"\phi^b f"]
& X
\ar[l] \ar[ur] \ar[d,"f"]
\\
& \phi^b Y \ar[dl]
& Y \ar[l] \ar[dr]
\\[-1.5em]
\phi^a \phi^b Y
\ar[rrr,"(\mu^{\phi}_{a,b})_Y"]
&&& \phi^{ab} Y
\end{tikzcd}
\qquad
\begin{tikzcd}
\phi^a \phi^b \phi^c X
\ar[rrr,"(\mu^{\phi}_{a,b})_{\phi^c X}"]
\ar[ddd,swap,"\phi^a (\mu^{\phi}_{b,c})_X" description]
&[-2em]&[-1em]&[-2em]
\phi^{ab} \phi^c X
\ar[ddd,"(\mu^{\phi}_{ab,c})_X" description]
\\[-1.5em]
&
\phi^b \phi^c X \ar[ul]
\ar[d,"(\mu^{\phi}_{b,c})_X"]
&
\phi^c X \ar[l] \ar[ur]
\\
&
\phi^{bc} X \ar[dl]
&
X \ar[u] \ar[l] \ar[dr]
\\[-1.5em]
\phi^a \phi^{bc} X
\ar[rrr,"(\mu^{\phi}_{a,bc})_X"]
&&&
\phi^{abc} X
\,.
\end{tikzcd}
\end{equation*}
Naturality of \(r_{X,1}\) is by \(\phi^1 f \of r_{X,1} = r_{Y,1} \of f \of r_{X,1}^{-1} \of r_{X,1}\), with coherence following directly from the definition of \(\mu^{\phi}\).
\end{proof}

\begin{lemma}
Let \(F \colon \cat{C} \to \cat{D}\) be an \(H\)-Hom-graded functor, and fix \(a \in H\).
Suppose \(\cat{C}\) and \(\cat{D}\) have fixed shifts inducing respective shift functors \(\phi\) and \(\psi\).
Then
\begin{equation*}
(s_F^a)_X = Fr_{X,a} \of r_{FX,a}^{-1}
\end{equation*}
defines an \(R\)-linear natural isomorphism \(s_F^a \colon \psi^a F^1 \tonat F^1 \phi^a\).
Moreover, for all \(a,b \in H\) there is a commuting diagram
\begin{equation*}
\begin{tikzcd}[column sep=3em]
\psi^a \psi^b  F^1
\ar[r,"\id \hof s_F^b"]
\ar[d,swap,"\mu^{\psi}_{a,b} \hof \id"]
&
\psi^a F^1 \phi^b
\ar[r,"s_F^a \hof \id"]
&
F^1 \phi^a \phi^b
\ar[d,"\id \hof \mu^{\phi}_{a,b}"]
\\
\psi^{ab} F^1
\ar[rr,swap,"s_F^{ab}"]
&&
F^1 \phi^{ab}
\end{tikzcd}
\end{equation*}
so that \((F^1,s)\) defines an \(H\)-module functor \((\cat{C}^1,\phi) \to (\cat{D}^1,\psi)\).
\label{lem:shiftaction-functor}
\end{lemma}
\begin{proof}
Naturality and coherence are by the commuting diagrams
\begin{equation*}
\begin{tikzcd}[column sep=.7em]
\psi^a FX
\ar[rr,"(s_F^a)_X"] \ar[ddd,swap,"\psi^a Ff" description]
&&
F \phi^a X
\ar[ddd,"F \phi^a f" description]
\\[-2em]
&
FX \ar[ul] \ar[ur]
\ar[d,"Ff"]
\\
&
FY \ar[dl] \ar[dr]
\\[-2em]
\psi^a FY
\ar[rr,swap,"(s_F^a)_Y"]
&&
F \phi^a Y
\end{tikzcd}
\qquad
\begin{tikzcd}
\psi^a \psi^b  FX
\ar[rr,"\psi^a (s_F^b)_X"]
\ar[ddd,swap,"(\mu^{\psi}_{a,b})_{FX}" description]
&[-2em]&[-2em]
\psi^a F \phi^b X
\ar[rr,"(s_F^a)_{\phi^b X}"]
&[-2em]&[-2em]
F \phi^a \phi^b X
\ar[ddd,"F (\mu^{\phi}_{a,b})_X" description]
\\[-1em]
&
\psi^b FX \ar[ul]
\ar[rr,"(s_F^b)_X"]
&&
F \phi^b X \ar[ul] \ar[ur]
\\[-2em]
&&
FX \ar[ul] \ar[ur] \ar[dll] \ar[drr]
\\[-2em]
\psi^{ab} FX
\ar[rrrr,swap,"(s_F^{ab})_X"]
&&&&
F \phi^{ab} X
\end{tikzcd}
\end{equation*}
for \(f\) of degree 1.
\end{proof}

\begin{proposition}
Then there is a 2-functor \((-)^1 \colon \Cat^{H,\shifts} \to \ModCat[H]\) sending
\begin{itemize}
\item each object \(\cat{C}\) to \((\cat{C}^1,\phi)\) as in \cref{prop:shiftaction}
\item each 1-morphism \(F\) to \((F^1,s_F)\) as in \cref{lem:shiftaction-functor}
\item each 2-morphism \(\eta\) to \(\eta^1\).
\end{itemize}
\label{prop:first2fun}
\end{proposition}
\begin{proof}
\Cref{prop:shiftaction,lem:shiftaction-functor} imply that this is well-defined on objects and 1-morphisms.
On 2-morphisms, naturality gives the commuting square
\begin{equation*}
\begin{tikzcd}
\psi^a EX
\ar[rr,"(s_E^a)_X"] \ar[ddd,swap,"\psi^a \eta_X"]
&& E \phi^a X
\ar[ddd,"\eta_{\phi^a X}"]
\\[-2em]
& EX \ar[ul] \ar[ur]
\ar[d,"\eta_X"]
\\
& FX \ar[dl] \ar[dr]
\\[-2em]
\psi^a FX
\ar[rr,swap,"(s_F^a)_X"]
&& F \phi^a X
\end{tikzcd}
\end{equation*}
so that \(\eta^1\) is indeed an \(H\)-module natural transformation and \((-)^1\) is well-defined.
We have \((s^{\id[\cat{C}]})_X = r_{X,a} \of r_{X,a}^{-1} = \id[\phi^a]\) so that \(\id[\cat{C}]^1 = \id[\cat{C}^1]\), and
\begin{equation*}
\begin{split}
(s_{F'} s_F)_X^a
&= F'(s_F^a)_X \of (s_{F'}^a)_{FX} \\
&= F'Fr_{X,a} \of F'r_{FX,a}^{-1} \of F'r_{FX,a} \of r_{F'FX,a}^{-1} \\
&= F'Fr_{X,a} \of \of r_{F'FX,a}^{-1} \\
&= (s_{F'F}^a)_X
\end{split}
\end{equation*}
so that \(F'^1 F^1 = (F'F)^1\), hence \((-)^1\) is a functor of the underlying 1-categories.
Since \((-)^1\) acts on 2-morphisms only by restricting the source and target, it automatically preserves vertical composition, horizontal composition and identity 2-morphisms.
Therefore \((-)^1\) is a 2-functor.
\end{proof}
\subsection{The 2-functor \texorpdfstring{\((-)^{\bullet} \colon \ModCat[H] \to \Cat^{H,\shifts}\)}{(-)•: H-ModCat → CatH}}
\label{sec:2fun2}
The reverse direction --- building \(H\)-Hom-graded categories from \(H\)-module categories --- is harder because we need to add new morphisms.
Inspired by the isomorphisms \(\Hom^h(X,Y) \iso \Hom^1(X\langle h \rangle, Y)\) in a Hom-graded category, we can keep track of morphism degrees by changing their source objects.
However, this means we must also change the way in which morphisms are composed.

\begin{definition}
Let \((\cat{C},\alpha)\) be an \(H\)-module category.
We write \((\cat{C},\alpha)^{\bullet}\) for the \(H\)-Hom-graded category with
\begin{itemize}
\item objects the same as \(\cat{C}\)
\item degree \(h\) morphisms given by \(\Hom_{(\cat{C},\alpha)^{\bullet}}^h(X,Y) = \Hom_{\cat{C}}(\alpha^h X, Y)\)
\item composition given by
 \begin{equation*}
 \begin{tikzcd}
 \alpha^{h'h} X
 \ar[rrr,"f' \of^{\bullet} f",rounded corners,to path={
   (\tikztostart.north east)
   -- ([yshift=.2em,xshift=1em]\tikztostart.north east)
   -- ([yshift=.2em,xshift=-1em]\tikztotarget.north west) \tikztonodes
   -- (\tikztotarget.north west)
 }]
 \ar[r,swap,"(\mu_{h',h})_X^{-1}"]
 &[2em]
 \alpha^{h'} \alpha^h X
 \ar[r,swap,"\alpha^{h'} f"]
 &
 \alpha^{h'} Y
 \ar[r,swap,"f'"]
 &[-1em]
 Z \vphantom{\alpha^{h'h} X}
 \end{tikzcd}
\end{equation*} 
for morphisms \(f \in \Hom_{(\cat{C},\alpha)^{\bullet}}^h(X,Y)\) and \(f' \in \Hom_{(\cat{C},\alpha)^{\bullet}}^{h'}(Y,Z)\)
\item identities given by \(\id[X]^{\bullet} = \epsilon_X^{-1} \in \Hom_{\cat{C}}(\alpha^1 X, X) = \Hom_{(\cat{C},\alpha)^{\bullet}}^1(X,X)\).
\end{itemize}
\label{def:2bullet-0}
\end{definition}

The composition \(\of^{\bullet}\) preserves morphism degrees by definition, and identitities are by
\begin{equation*}
\begin{tikzcd}
\alpha^h X
\ar[r,shift left,"(\mu_{1,h})_X^{-1}"]
\ar[rrr, "\epsilon_Y^{-1} \of^{\bullet} f", rounded corners, to path={
  (\tikztostart.north)
  -- ([yshift=1.2em]\tikztostart.north)
  -- ([yshift=1.2em]\tikztotarget.north) \tikztonodes
  -- (\tikztotarget.north)
}]
\ar[rrr, swap, "f", rounded corners, to path={
  (\tikztostart.south)
  -- ([yshift=-1.2em]\tikztostart.south)
  -- ([yshift=-1.2em]\tikztotarget.south) \tikztonodes
  -- (\tikztotarget.south)
}]
&[1em]
\alpha^1 \alpha^h X
\ar[l,shift left,"\epsilon_{\alpha^h X}^{-1}"]
\ar[r,"\alpha^1 f"]
& \alpha^1 Y
\ar[r,"\epsilon_Y^{-1}"]
& Y
\end{tikzcd}
\end{equation*}
\begin{equation*}
\begin{tikzcd}
\alpha^h X
\ar[r,shift left,"(\mu_{h,1})_X^{-1}"]
\ar[rrr, "f \of^{\bullet} \epsilon_X^{-1}", rounded corners, to path={
  (\tikztostart.north)
  -- ([yshift=1.2em]\tikztostart.north)
  -- ([yshift=1.2em]\tikztotarget.north) \tikztonodes
  -- (\tikztotarget.north)
}]
\ar[rrr, swap, "f", rounded corners, to path={
  (\tikztostart.south)
  -- ([yshift=-1.2em]\tikztostart.south)
  -- ([yshift=-1.2em]\tikztotarget.south) \tikztonodes
  -- (\tikztotarget.south)
}]
&[1em]
\alpha^h \alpha^1 X
\ar[r,shift left,"\alpha^h \epsilon_X^{-1}"]
&[1em]
\alpha^h X
\ar[r,"f"]
\ar[l,shift left,"(\mu^{-1}_{h,1})_X"]
& Y
\,.
\end{tikzcd}
\end{equation*}
For homogeneous morphisms \(f'', f', f\) of respective degrees \(a,b,c\), associativity is by the commutative diagram
\begin{equation*}
\begin{tikzcd}
&[1em]
\alpha^{ab} \alpha^c X
\ar[d,"(\mu_{a,b})_{\alpha^h X}^{-1}"] \ar[r,"\alpha^{ab} f"]
&[4em]
\alpha^{ab} Y
\ar[dr,near start,"f'' \of^{\bullet} f'"] \ar[d,swap,"(\mu_{a,b})_Y^{-1}"]
&[1em]
\\
\alpha^{abc} X
\ar[rrr, "(f'' \of^{\bullet} f') \of^{\bullet} f", rounded corners, to path={
  ([xshift=-.5em]\tikztostart.north)
  -- ([yshift=4.5em,xshift=-.5em]\tikztostart.north)
  -- ([yshift=4.5em]\tikztotarget.north) \tikztonodes
  -- (\tikztotarget.north)
}]
\ar[rrr, swap, "f'' \of^{\bullet} (f' \of^{\bullet} f)", rounded corners, to path={
  ([xshift=-.5em]\tikztostart.south)
  -- ([yshift=-4.5em,xshift=-.5em]\tikztostart.south)
  -- ([yshift=-4.5em]\tikztotarget.south) \tikztonodes
  -- (\tikztotarget.south)
}]
\ar[ur,near end,"(\mu_{ab,c})_X^{-1}"]
\ar[dr,near end,swap,"(\mu_{a,bc})_X^{-1}"]
& \alpha^a \alpha^b \alpha^c X
\ar[r,"\alpha^a \alpha^b f"]
& \alpha^a \alpha^b Y
\ar[d,swap,"\alpha^a f'"]
& W
\\
& \alpha^a \alpha^{bc} X
\ar[u,swap,"\alpha^a (\mu_{b,c})_X^{-1}"] \ar[r,swap,"\alpha^a (f' \of^{\bullet} f)"]
& \alpha^a Z
\ar[ur,near start,swap,"f''"]
\,.
\end{tikzcd}
\end{equation*}

\begin{lemma}
For each \(H\)-module category \((\cat{C},\alpha)\), the \(H\)-Hom-graded category \((\cat{C},\alpha)^{\bullet}\) has shifts given by \(X\langle a \rangle = \alpha^a X\) together with the isomorphisms
\begin{align*}
r^{\alpha}_{X,a}
&= \id[\alpha^a X]
&\in\ &
\Hom_{\cat{C}}(\alpha^a X, \alpha^a X)
= \Hom_{(\cat{C},\alpha)^{\bullet}}^a(X, \alpha^a X)
\\
(r^{\alpha}_{X,a})^{-1}
&= \epsilon_X^{-1} \of (\mu_{a^{-1},a})_X
&\in\ &
\Hom_{\cat{C}}(\alpha^{a^{-1}} \alpha^a X, X)
= \Hom_{(\cat{C},\alpha)^{\bullet}}^{a^{-1}}(\alpha^a X, X)
\,.
\end{align*}
\label{lem:2bullet-shifts}
\end{lemma}
\begin{proof}
These morphisms have the correct degrees, sources and targets by definition.
Invertibility is by the commutative diagrams
\begin{equation*}
\begin{tikzcd}[row sep=1em]
\alpha^1 X
\ar[rrrr,"{
  (\epsilon_X^{-1} \of (\mu_{a^{-1},a})_X) \of^{\bullet} \id[\alpha^a X]
}",rounded corners,to path={
  (\tikztostart.north)
  -- ([yshift=1.2em]\tikztostart.north)
  -- ([yshift=1.2em]\tikztotarget.north) \tikztonodes
  -- (\tikztotarget.north)
}]
\ar[rrrr,swap,"{
  \epsilon_X^{-1} \ =\ \id[X]^{\bullet}
}",rounded corners,to path={
  (\tikztostart.south)
  -- ([yshift=-.5em]\tikztostart.south)
  -- ([yshift=-.5em]\tikztotarget.south) \tikztonodes
  -- (\tikztotarget.south)
}]
\ar[r,"(\mu_{a^{-1},a})_X^{-1}"]
&[3em]
\alpha^{a^{-1}} \alpha^a X
\ar[r,"{\alpha^{a^{-1}} \id[\alpha^a X]}"]
&[3em]
\alpha^{a^{-1}} \alpha^a X
\ar[r,"(\mu_{a^{-1},a})_X^{-1}"]
&[2em]
X
\ar[r,"\epsilon_X^{-1}"]
&
\alpha^a X
\end{tikzcd}
\end{equation*}
\begin{equation*}
\begin{tikzcd}[row sep=1em]
\alpha^1 \alpha^a X
\ar[rrrr,"{
  \id[\alpha^a X] \of^{\bullet} (\epsilon_X^{-1} \of (\mu_{a^{-1},a})_X)
}",rounded corners,to path={
  (\tikztostart.north)
  -- ([yshift=1.2em]\tikztostart.north)
  -- ([yshift=1.2em]\tikztotarget.north) \tikztonodes
  -- (\tikztotarget.north)
}]
\ar[rrrr,swap,"{
  (\mu_{1,a})_X \ =\ \epsilon_{\alpha^a X}^{-1} \ =\ \id[\alpha^a X]^{\bullet}
}",rounded corners,to path={
  (\tikztostart.south)
  -- ([yshift=-1.2em]\tikztostart.south)
  -- ([yshift=-1.2em]\tikztotarget.south) \tikztonodes
  -- (\tikztotarget.south)
}]
\ar[r,"(\mu_{a,a^{-1}})_{\alpha^a X}^{-1}"]
&[3em]
\alpha^a \alpha^{a^{-1}} \alpha^a X
\ar[r,"\alpha^a (\mu_{a^{-1},a})_X"]
&[3em]
\alpha^a \alpha^1 X
\ar[r,shift left,"\alpha^a \epsilon_X^{-1}"]
\ar[r,shift right,swap,"\alpha^a (\mu_{a,1})_X"]
&[2em]
\alpha^a X
\ar[r,swap,"\id"]
&
\alpha^a X
\,.
\end{tikzcd}
\end{equation*}
\end{proof}

\begin{lemma}
Let \((F,s) \colon (\cat{C},\alpha) \to (\cat{D},\beta)\) be an \(H\)-module functor.
Then there is an \(H\)-Hom-graded functor \((F,s)^{\bullet} \colon (\cat{C},\alpha)^{\bullet} \to (\cat{D},\beta)^{\bullet}\) mapping each object \(X \mapsto FX\) and each morphism \(f \in \Hom_{(\cat{C},\alpha)^{\bullet}}^h(X,Y) = \Hom_{\cat{C}}(\alpha^h X,Y)\) to
\begin{equation*}
(F,s)^{\bullet} f
= Ff \of s^h_X
\in \Hom_{\cat{D}}(\beta^h FX,FY)
\in \Hom_{(\cat{D},\beta)^{\bullet}}^h(FX,FY)
\,.
\end{equation*}
\label{lem:2bullet-1}
\end{lemma}
\begin{proof}
By definition, \((F,s)^{\bullet}\) preserves morphism degrees.
Functoriality is by the commutative diagrams
\begin{equation*}
\qquad\qquad\qquad 
\begin{tikzcd}[column sep=5em]
\beta^1 FX
\ar[r,"{(F,s)^{\bullet} \id[X]^{\bullet}}"]
\ar[d,swap,"s^1_X"]
& FX
\\
F\alpha^1 X
\ar[r,swap,"(s^1_X)^{-1}"]
\ar[ur,"{F\id[X]^{\bullet} = F(\epsilon^{\alpha}_X)^{-1}}" description]
&
\beta^1 FX
\ar[u,swap,"{(\epsilon^{\beta}_{FX})^{-1} = \id[FX]^{\bullet}}"]
\end{tikzcd}
\end{equation*}
\begin{equation*}
\begin{tikzcd}
\beta^{h'h} FX
\ar[rr,"{(F,s)^{\bullet} (f' \of^{\bullet} f)}"]
\ar[dr,"s^{h'h}_X"]
\ar[d,swap,"(\mu^{\beta}_{h',h})_{FX}^{-1}"]
&& FZ
\\
\beta^{h'} \beta^h FX
\ar[d,"\beta^{h'} s^h_X"]
\ar[ddr, swap, "{\beta^{h'} (F,s)^{\bullet} f}", rounded corners=.8em, to path={
  (\tikztostart.west)
  -- ([xshift=-.5em]\tikztostart.west)
  -- ([xshift=-.5em,yshift=-7.3em]\tikztostart.west) \tikztonodes
  -- (\tikztotarget.west)
}]
& F \alpha^{h'h} X
\ar[d,"F(\mu^{\alpha}_{h',h})_X^{-1}"]
\ar[ur,"F(f' \of^{\bullet} f)"]
\\
\beta^{h'} F \alpha^h X
\ar[r,"s^{h'}_{\alpha^h X}"]
\ar[dr,swap,"\beta^{h'} Ff"]
& F \alpha^{h'} \alpha^h X
\ar[r,"F\alpha^{h'} f"]
& F \alpha^{h'} Y
\ar[uu,"Ff"]
\\
& \beta^{h'} FY
\ar[ur,swap,"s^{h'}_Y"]
\ar[uuur, swap, "{(F,s)^{\bullet} f'}", rounded corners=.8em, to path={
  (\tikztostart.east)
  -- ([xshift=1em,yshift=-10.9em]\tikztotarget.east)
  -- ([xshift=1em]\tikztotarget.east) \tikztonodes
  -- (\tikztotarget.east)
}]
\end{tikzcd}
\end{equation*}
since the bottom path of the latter is precisely the definition of \((F,s)^{\bullet} f' \of^{\bullet} (F,s)^{\bullet} f\).
\end{proof}

\begin{lemma}
Let \((E,r), (F,s) \colon (\cat{C},\alpha) \to (\cat{D},\beta)\) be \(H\)-module functors with an \(H\)-module natural transformation \(\eta \colon (E,r) \tonat (F,s)\).
Then there is an \(H\)-Hom-graded natural transformation \(\eta^{\bullet} \colon (E,r)^{\bullet} \to (F,s)^{\bullet}\) with components given by
\begin{equation*}
\eta^{\bullet}_X
= \eta_X \of (\epsilon^{\beta}_{EX})^{-1}
\in \Hom_{\cat{D}}(\beta^1 EX, FX)
\in \Hom_{(\cat{D},\beta)^{\bullet}}^1(EX, FX)
\,.
\end{equation*}
\label{lem:2bullet-2}
\end{lemma}
\begin{proof}
Let \(f \in \Hom_{(\cat{C},\alpha)^{\bullet}}^h(X,Y) = \Hom_{\cat{C}}(\alpha^h X, Y)\).
Expanding the composition rule gives
\begin{gather*}
\eta^{\bullet}_Y \of^{\bullet} (E,r)^{\bullet} f
= (\eta_Y \of (\epsilon^{\beta}_{EY})^{-1}) \of^{\bullet} (Ef \of r^h_X)
= \eta_Y \of (\epsilon^{\beta}_{EY})^{-1} \of \beta^1 Ef \of \beta^1 r^h_X \of (\mu^{\beta}_{1,h})_{EX}^{-1}
\\
(F,s)^{\bullet} f \of^{\bullet} \eta^{\bullet}_X
= (Ff \of s^h_X) \of^{\bullet} (\eta_X \of (\epsilon^{\beta}_{EX})^{-1})
= Ff \of s^h_X \of \beta^h \eta_X \of \beta^h (\epsilon^{\beta}_{EX})^{-1}
\of (\mu_{h,1}^{\beta})_{EX}^{-1}
\,,
\end{gather*}
and these two lines are equal by the commuting diagram
\begin{equation*}
\begin{tikzcd}[column sep=4em]
&[-3em] \beta^1 \beta^h EX
\ar[r,"\beta^1 r^h_X"]
\ar[rdd,"(\epsilon^{\beta}_{\beta^h EX})^{-1}" description]
& \beta^1 E\alpha^h X
\ar[r,"\beta^1 Ef"]
\ar[rd,"(\epsilon^{\beta}_{E\alpha^h X})^{-1}" description]
& \beta^1 EY
\ar[r,"(\epsilon^{\beta}_{EY})^{-1}"]
& EY
\ar[rd,"\eta_Y"]
&[-3em]
\\
\beta^h EX
\ar[ru,"(\mu^{\beta}_{1,h})_{EX}^{-1}"]
\ar[rd,swap,"(\mu^{\beta}_{h,1})_{EX}^{-1}"]
\ar[rrd,"\id"]
&&& E\alpha^h X
\ar[ru,"Ef" description]
\ar[rd,"\eta_{\alpha^h X}"]
&& FY
\\
& \beta^h \beta^1 EX
\ar[r,swap,"\beta^h(\epsilon^{\beta}_{EX})^{-1}"]
& \beta^h EX
\ar[r,swap,"\beta^h \eta_X"]
\ar[ru,"r^g_X"]
& \beta^h FX
\ar[r,swap,"s^g_X"]
& F\alpha^h X
\ar[ru,swap,"Ff"]
\,.
\end{tikzcd}
\end{equation*}
\end{proof}

\begin{proposition}
There is a 2-functor \((-)^{\bullet} \colon \ModCat[H] \to \Cat^{H,\shifts}\) defined by \cref{def:2bullet-0,lem:2bullet-1,lem:2bullet-2}.
\end{proposition}
\begin{proof}
This is well-defined by \cref{lem:2bullet-shifts}.
Let \(f \in \Hom_{(\cat{C},\alpha)^{\bullet}}^h(X,Y)\).
Then \((-)^{\bullet}\) preserves identity 1-morphisms by
\begin{equation*}
\id[(\cat{C},\alpha)]^{\bullet} X
= \id[\cat{C}] X
= X
\qquad
\id[(\cat{C},\alpha)]^{\bullet} f
= \id[\cat{C}] f \of \id[\alpha^h X]
= f
\end{equation*}
and the composition of 1-morphisms by
\begin{gather*}
((E,r) \of (F,s))^{\bullet} X
= EFX
= E^{\bullet} F^{\bullet} X
\\
\begin{split}
((E,r) \of (F,s))^{\bullet} f
&= (EF, E(s^h_X \of r^h_{FX}))^{\bullet} f \\
&= EFf \of E s^h_X \of r^h_{FX} \\
&= E^{\bullet} (Ff \of s^h_X) \\
&= E^{\bullet} F^{\bullet} f \,.
\end{split}
\end{gather*}
Identity 2-morphisms are preserved by
\begin{equation*}
(\id[(F,s)])^{\bullet}_X
= \id[FX] \of (\epsilon^{\beta}_{FX})^{-1}
= \id[FX]^{\bullet}
= (\id[(F,s)^{\bullet}]^{\bullet})_X
\end{equation*}
for \((F,s) \colon (\cat{C},\alpha) \to (\cat{D},\beta)\).
Given a diagram in \(\ModCat[H]\) of the form
\begin{equation*}
\begin{tikzcd}[column sep=4em]
(\cat{A},\alpha)
\ar[r,bend left,"{(E,r)}"{name=e}]
\ar[r,"{(F,s)}" {description, name=f}]
\ar[r,bend right,swap,"{(H,t)}" {name=h}]
\nat[shorten <=2pt, shorten >=1pt, from=e, to=f, "\,\nu"]
\nat[shorten <=1pt, shorten >=2pt, from=f, to=h, "\,\eta"]
& (\cat{B},\beta)
\ar[r,bend left,"{(E',r')}"{name=E}]
\ar[r,bend right,swap,"{(F',s')}"{name=F}]
\nat[shorten <=3pt, shorten >=3pt, from=E, to=F, "\,\nu'"]
& (\cat{C},\gamma)
\,,
\end{tikzcd}
\end{equation*}
we have preservation of vertical and horizontal composition by the commutative diagrams
\begin{equation*}
\begin{tikzcd}[row sep=1em]
\beta^1 EX
\ar[rrrr,"(\eta \of \nu)^{\bullet}_X"]
\ar[dr,"\id"]
\ar[ddd,swap,"(\mu^{\beta}_{1,1})_{EX}^{-1}"]
&[1em]&&[-2em]&
HX
\\[-.5em]
& \beta^1 EX
\ar[r,"(\epsilon^{\beta}_{EX})^{-1}"]
\ar[drrr,"\beta^1 \nu_X"]
& EX
\ar[r,"\nu_X"]
& FX
\ar[ur,"\eta_X"]
\\
&& \beta^1 \beta^1 FX
\ar[rr,swap,near start,"\beta^1 (\epsilon^{\beta}_{FX})_{-1}"]
&& \beta^1 FX
\ar[ul,swap,near end,"(\epsilon^{\beta}_{FX})^{-1}"]
\ar[uu,swap,"\eta^{\bullet}_X"]
\\
\beta^1 \beta^1 EX
\ar[uur,near end,"\beta^1 (\epsilon^{\beta}_{EX})^{-1}"]
\ar[urr,near end,"\beta^1 \beta^1 \nu_X"]
\ar[rrr,near end,"\beta^1 (\epsilon^{\beta}_{EX})^{-1}"]
\ar[urrrr,near end,swap,"\beta^1 \nu_X",rounded corners,to path={
  ([xshift=1em]\tikztostart.south)
  -- ([yshift=-.5em,xshift=2em]\tikztostart.south)
  -- ([yshift=-3.2em]\tikztotarget.south)
  -| (\tikztotarget.south) \tikztonodes
}]
&&& \beta^1 EX
\ar[ur,swap,near start,"\beta^1 \nu_X"]
\end{tikzcd}
\end{equation*}
\begin{equation*}
\begin{tikzcd}[column sep=4em]
\gamma^1 E'EX
\ar[rrr,"(\nu'^{\bullet} \hof \nu^{\bullet})_X \ =\ \nu'^{\bullet}_{FX} \of^{\bullet} E'^{\bullet} \nu^{\bullet}_X"]
\ar[dd,swap,"\id"]
&&& F'FX
\\[-1.5em]
& \gamma^1 \gamma^1 E'EX
\ar[dr,"{\gamma^1 {}E'^{\bullet} \nu^{\bullet}_X}"]
\ar[d,"\gamma^1 r'^1_{EX}"]
\\
\gamma^1 E'EX
\ar[ur,near end,"(\mu^{\gamma}_{1,1})_{E'EX}^{-1}"]
\ar[r,swap,"\gamma^1 E' \epsilon^{\beta}_{EX}"]
\ar[rr,swap,"\gamma^1 E' \nu_X",rounded corners,to path={
  ([xshift=1em]\tikztostart.south)
  -- ([yshift=-1.2em,xshift=2em]\tikztostart.south)
  -- ([yshift=-1.2em,xshift=-2em]\tikztotarget.south) \tikztonodes
  -- ([xshift=-1em]\tikztotarget.south)
}]
\ar[dd,swap,"(\epsilon^{\gamma}_{E'EX})^{-1}"]
& \gamma^1 E' \beta^1 EX
\ar[r,swap,"\gamma^1 E' \nu^{\bullet}_X"]
& \gamma^1 E'FX
\ar[r,"\nu'^{\bullet}_{FX}"]
\ar[d,"(\epsilon^{\gamma}_{E'FX})^{-1}"]
& F'FX
\ar[uu,swap,"\id"]
\\
&& E'FX
\ar[ur,bend right=1em,swap,"\nu'_{FX}"]
\\[-2.1em]
E'EX
\ar[urr,shift left=.1em,near start,"E' \nu_X"]
\ar[uurrr,swap,near end,"(\nu' \hof \nu)_X",rounded corners,to path={
  (\tikztostart.east)
  -- ([yshift=-4em]\tikztotarget.south)
  -| (\tikztotarget.south) \tikztonodes
}]
\,.
\end{tikzcd}
\end{equation*}
\end{proof}
\subsection{2-equivalence between \texorpdfstring{\(\Cat^{H,\shifts}\)}{CatH<>} and \texorpdfstring{\(\ModCat[H]\)}{H-ModCat}}
\label{sec:2eqv1}
It remains to show that the 2-functors \((-)^1\) and \((-)^{\bullet}\) form a 2-equivalence.
That is, we must find 2-natural isomorphisms
\begin{equation*}
\two{\nu} \colon ((-)^{\bullet})^1 \to \iid[\ModCat[H]]
\qquad
\two{\eta} \colon ((-)^1)^{\bullet} \to \iid[\Cat^{H,\shifts}]
\,.
\end{equation*}

\begin{lemma}
Let \((\cat{C},\alpha)\) be an \(H\)-module category.
Then the \(H\)-action on \(((\cat{C},\alpha)^{\bullet})^1\) is given by \(\phi \colon \Disc(H) \to \Aut(\cat(C)^1)\) with
\begin{align*}
\phi^h X &= \alpha^h X &
\epsilon^{\phi}_X &= r^{\alpha}_{X,1} = \id[\alpha^1 X] \\
\phi^h f &= \alpha^h f \of \alpha^h \epsilon^{\alpha}_X \of (\epsilon^{\alpha}_{\alpha^h X})^{-1} &
(\mu^{\phi}_{a,b})_X &= (\mu^{\alpha}_{a,b})_X \of (\epsilon^{\alpha}_{\phi^a \phi^b X})^{-1}
\end{align*}
for morphisms \(f \in \Hom_{((\cat{C},\alpha)^{\bullet})^1}(X,Y) = \Hom_{\cat{C}}(\alpha^1 X, Y)\).
\end{lemma}
\begin{proof}
Applying the 2-functors gives
\begin{align*}
\phi^h X &= \alpha^h X &
\epsilon^{\phi}_X &= r^{\alpha}_{X,1} = \id[\alpha^1 X] \\
\phi^h f
&= r^{\alpha}_{Y,h} \of^{\bullet} f \of^{\bullet} (r^{\alpha}_{X,h})^{-1} &
(\mu^{\phi}_{a,b})_X
&= r^{\alpha}_{X,ab} \of^{\bullet} (r^{\alpha}_{X,b})^{-1} \of^{\bullet} (r^{\alpha}_{\alpha^b X,a})^{-1} \,.
\end{align*}
The commuting diagram
\begin{equation*}
\begin{tikzcd}[column sep=4em]
\alpha^1 \alpha^h X
\ar[rr,
"r^{\alpha}_{Y,h} \of^{\bullet} (f \of^{\bullet} (r^{\alpha}_{X,h})^{-1})"
]
\ar[drr, near end, swap, "(\mu^{\alpha}_{h,h^{-1}})^{-1}_{\alpha^h X}"]
\ar[d, swap, "(\epsilon^{\alpha}_{\alpha^h X})^{-1} = (\mu^{\alpha}_{1,h})_X"]
&[-1.5em]&
\alpha^h Y
&[-2em]&[-5em]
\alpha^h Y
\ar[ll,swap,"r^{\alpha}_{Y,h} = \id"]
\\
\alpha^h X
\ar[d, swap, "\alpha^h \epsilon^{\alpha}_X = (\mu^{\alpha}_{h,1})^{-1}_X"]
&& \alpha^h \alpha^{h^{-1}} \alpha^h X
\ar[urr, near start, swap, "\alpha^h f \of^{\bullet} \alpha^h (r^{\alpha}_{X,h})^{-1}"]
\ar[dr, swap, "\alpha^h (\mu^{\alpha}_{1,h^{-1}})^{-1}_{\alpha^h X}"]
&& \alpha^h \alpha^1 X
\ar[u,swap,"\alpha^h f"]
\\[1em]
\alpha^h \alpha^1 X
\ar[urr, swap, near start, "\alpha^h (\mu^{\alpha}_{h^{-1},h})^{-1}_X"]
\ar[drr, swap, "\alpha^h (\mu_{1,1}^{\alpha})^{-1}_X"]
\ar[urrrr, swap, rounded corners, "\id", to path={
  (\tikztostart.south)
  -- ([yshift=-3.3em]\tikztostart.south) \tikztonodes
  -- ([yshift=-7.8em]\tikztotarget.south)
  -- (\tikztotarget.south)
}]
&&& \alpha^h \alpha^1 \alpha^{h^{-1}} \alpha^h X
\ar[ur, "\alpha^h \alpha^1 (r^{\alpha}_{X,h})^{-1}"]
\ar[dl, swap,"\alpha^h \alpha^1 (\mu^{\alpha}_{h^{-1},h})_X"]
\\[-.5em]
&& \alpha^h \alpha^1 \alpha^1 X
\ar[uurr, rounded corners, "\alpha^h \alpha^1 (\epsilon^{\alpha}_X)^{-1}", to path={
  (\tikztostart.east)
  -- ([xshift=10.3em]\tikztostart.east) \tikztonodes
  -- ([xshift=-.4em]\tikztotarget.south)
}]
\end{tikzcd}
\end{equation*}
gives \(\phi^h f = \alpha^h f \of \alpha^h \epsilon^{\alpha}_X \of (\epsilon^{\alpha}_{\alpha^h X})^{-1}\).
The commuting diagram
\begin{equation*}
\begin{tikzcd}
\alpha^{(ab)^{-1}} \alpha^a \alpha^b X
\ar[rrrr,
"(r^{\alpha}_{X,b})^{-1} \of^{\bullet} (r^{\alpha}_{\alpha^b X, a})^{-1}"
]
\ar[drr,"(\mu^{\alpha}_{b^{-1},a^{-1}})^{-1}_{\alpha^a \alpha^b X}"]
\ar[ddd,swap,"(\mu^{\alpha}_{(ab)^{-1},a})_{\alpha^b X}"]
&&[-6em]&[-1em]&[2em] X
\\
&& \alpha^{b^{-1}} \alpha^{a^{-1}} \alpha^a \alpha^b X
\ar[ddr,"\alpha^{b^{-1}} (r^{\alpha}_{\alpha^h X, a})^{-1}"]
\ar[dl,swap,near end,"\alpha^{b^{-1}} (\mu^{\alpha}_{a^{-1},a})_{\alpha^b X}"]
\\
& \alpha^{b^{-1}} \alpha^1 \alpha^b X
\ar[dl,swap,near start,"(\mu_{b^{-1},1})_{\alpha^b X}"]
\ar[drr,swap,near start,"\alpha^{b^{-1}} (\epsilon^{\alpha}_{\alpha^b X})^{-1}"]
\\
\alpha^{b^{-1}} \alpha^b X
\ar[rrr,swap,"\id"]
&&& \alpha^{b^{-1}} \alpha^b X
\ar[uuur,"(r^{\alpha}_{X,b})^{-1}"]
\ar[r,swap,"(\mu^{\alpha}_{b^{-1},b})_X"]
& \alpha^1 X
\ar[uuu,swap,"(\epsilon^{\alpha}_X)^{-1}"]
\end{tikzcd}
\end{equation*}
makes \(\alpha^{ab} ((r^{\alpha}_{X,b})^{-1} \of^{\bullet} (r^{\alpha}_{\alpha^b X,a})^{-1})\) match the downward right hand side of the further commuting diagram
\begin{equation*}
\begin{tikzcd}
\alpha^1 \alpha^a \alpha^b X
\ar[rr,"(\mu^{\alpha}_{ab,(ab)^{-1}})^{-1}_{\alpha^a \alpha^b X}"]
\ar[dr,"(\mu^{\alpha}_{1,a})_{\alpha^b X}"]
\ar[ddd,swap,"\alpha^1 (\mu^{\alpha}_{a,b})_X"]
&&[1em] \alpha^{ab} \alpha^{(ab)^{-1}} \alpha^a \alpha^b X
\ar[d,"\alpha^{ab} (\mu^{\alpha}_{(ab)^{-1},a})_{\alpha^b X}"]
\\
& \alpha^a \alpha^b X
\ar[d,"(\mu^{\alpha}_{a,b})_X"]
& \alpha^{ab} \alpha^{b^{-1}} \alpha^b X
\ar[d,"\alpha^{ab} (\mu^{\alpha}_{b^{-1},b})_X"]
\ar[l,swap,"(\mu^{\alpha}_{ab,b^{-1}})_{\alpha^b X}"]
\\
& \alpha^{ab} X
& \alpha^{ab} \alpha^1 X
\ar[l,swap,"(\mu^{\alpha}_{ab,1})_X"]
\\
\alpha^1 \alpha^{ab} X
\ar[ur,"(\mu^{\alpha}_{1,ab})_X"]
\ar[rr,swap,"(\epsilon^{\alpha}_{\alpha^{ab} X})^{-1}"]
&& \alpha^{ab} X
\ar[ul,"\id"]
\ar[u,swap,"\alpha^{ab} \epsilon^{\alpha}_X"]
\end{tikzcd}
\end{equation*}
to give
\begin{equation*}
\begin{split}
(\mu^{\phi}_{a,b})_X
&= r^{\alpha}_{X,ab} \of^{\bullet} (r^{\alpha}_{X,b})^{-1} \of^{\bullet} (r^{\alpha}_{\alpha^b X, a})^{-1} \\
&= r^{\alpha}_{X,ab}
\of \phi^{ab} ((r^{\alpha}_{X,b})^{-1} \of^{\bullet} (r^{\alpha}_{\phi^b X, a})^{-1})
\of (\mu^{\alpha}_{ab,(ab)^{-1}})^{-1}_{\phi^a \phi^b X} \\
&= \id[\phi^{ab} X] \of (\epsilon^{\alpha}_{\phi^{ab} X})^{-1}
\of \alpha^1 (\mu_{a,b}^{\alpha})_X \\
&= (\mu^{\alpha}_{a,b})_X \of (\epsilon^{\alpha}_{\phi^a \phi^b X})^{-1}
\,.\qedhere
\end{split}
\end{equation*}
\end{proof}

\begin{lemma}
For each \(H\)-module category \((\cat{C},\alpha)\), there is an \(H\)-module equivalence
\begin{equation*}
\two{\nu}_{(\cat{C},\alpha)} = (\two{\nu}_{(\cat{C},\alpha)}, \id) \colon ((\cat{C},\alpha)^{\bullet})^1 \to (\cat{C},\alpha)
\end{equation*}
mapping objects \(X \mapsto X\) and morphisms \(f \in \Hom_{((\cat{C},\alpha)^{\bullet})^1}(X,Y) = \Hom_{\cat{C}}(\alpha^1 X,Y)\) to \(f \of \epsilon^{\alpha}_X\).
\label{lem:2eqv-nu}
\end{lemma}
\begin{proof}
A strict inverse can be given mapping morphisms \(f \mapsto f \of (e^{\alpha}_X)^{-1}\).
For degree 1 morphisms \(X \to[f] Y \to[f'] Z\) we have
\begin{equation*}
f' \of^{\bullet} f
= f' \of \alpha^1 f \of (\mu_{1,1}^{\alpha})^{-1}_X
= f' \of \alpha^1 f \of \epsilon^{\alpha}_{\alpha^1 X}
= f' \of \epsilon^{\alpha}_Y \of f
\,,
\end{equation*}
so that \(\two{\nu}_{(\cat{C},\alpha)}^1\) has functoriality given by
\begin{gather*}
\two{\nu}_{(\cat{C},\alpha)} \id[X]^{\bullet}
= \id[X]^{\bullet} \of \epsilon^{\alpha}_X
= (\epsilon^{\alpha}_X)^{-1} \of \epsilon^{\alpha}_X
= \id[X]
= \id[\two{\nu}_{(\cat{C},\alpha)} X]
\\
\two{\nu}_{(\cat{C},\alpha)} (f' \of^{\bullet} f)
= f' \of \epsilon^{\alpha}_Y \of f \of \epsilon^{\alpha}_X
= \two{\nu}_{(\cat{C},\alpha)} f' \of \two{\nu}_{(\cat{C},\alpha)} f
\,.
\end{gather*}
Now \(\id[\alpha^h \alpha]\) gives a natural transformation \(\alpha^h \two{\nu}_{(\cat{C},\alpha)} \tonat \two{\nu}_{(\cat{C},\alpha)} \phi^h\) by
\begin{equation*}
\two{\nu}_{(\cat{C},\alpha)} \phi^h f
= \phi^h f \of \epsilon^{\alpha}_{\alpha^h X}
= \alpha^h f \of \alpha^h \epsilon^{\alpha}_X \of (\epsilon^{\alpha}_{\alpha^h X})^{-1} \of \epsilon^{\alpha}_{\alpha^h X}
= \alpha^h \two{\nu}_{(\cat{C},\alpha)} f \,,
\end{equation*}
and the \(H\)-module equivalence coherence diagrams follow directly from
\begin{gather*}
\two{\nu}_{(\cat{C},\alpha)} \epsilon^{\phi}_X
= \epsilon^{\phi}_X \of \epsilon^{\alpha}_X
= \id[X] \of \epsilon^{\alpha}_X
= \epsilon^{\alpha}_X
\\
\two{\nu}_{(\cat{C},\alpha)} (\mu^{\phi}_{a,b})_X
= (\mu^{\phi}_{a,b})_X \of \epsilon^{\alpha}_{\phi^a \phi^b X}
= (\mu^{\alpha}_{a,b})_X \of (\epsilon^{\alpha}_{\phi^a \phi^b X})^{-1} \epsilon^{\alpha}_{\phi^a \phi^b X}
= (\mu^{\alpha}_{a,b})_X
\,.\qedhere
\end{gather*}
\end{proof}

\begin{lemma}
Let \(\cat{C}\) be an \(H\)-Hom-graded category with shifts denoted \(r_{X,a} \colon X \to X\langle a \rangle\).
Then there is an \(H\)-Hom-graded equivalence \(\two{\eta}_{\cat{C}} \colon (\cat{C}^1)^{\bullet} \to \cat{C}\) mapping objects \(X \mapsto X\) and morphisms \(f \in \Hom_{(\cat{C}^1)^{\bullet}}^h(X,Y) = \Hom_{\cat{C}}^1(X\langle h \rangle, Y)\) to \(f \of r_{X,h}\).
\label{lem:2eqv-eta}
\end{lemma}
\begin{proof}
A strict inverse can be given mapping morphisms \(f \mapsto f \of r_{X,h}^{-1}\).
Functoriality is by
\begin{equation*}
\two{\eta}_{\cat{C}} \id[X]^{\bullet}
= \id[X]^{\bullet} \of r_{X,1}
= \epsilon_X^{-1} \of e_X
= \id[X]
\end{equation*}
and the commutative diagram
\begin{equation*}
\begin{tikzcd}
X
\ar[dr] \ar[dd] \ar[rr,"\two{\eta}_{\cat{C}}(f' \of^{\bullet} f)"]
\ar[ddd,bend right=6em,swap,"\two{\eta}_{\cat{C}} f"]
&& Z
\\[-1.5em]
& \phi^{h'h} X
\ar[d,"(\mu^{\phi}_{h',h})_X"] \ar[ur,swap,near start,"f' \of^{\bullet} f"]
\\
\phi^h X
\ar[r] \ar[d,swap,"f"]
& \phi^{h'} \phi^h X
\ar[r,swap,"\phi^{h'} f"]
& \phi^{h'} Y
\ar[uu,swap,"f'"]
\\
Y
\ar[rr,"\id"]
&& Y
\ar[u] \ar[uuu,bend right=6em,swap,"\two{\eta}_{\cat{C}} f'"]
\,.
\end{tikzcd}
\end{equation*}
\end{proof}

\begin{theorem}
The functors \((-)^1 \colon \Cat^{H,\shifts} \rightleftarrows \ModCat[H] \colon (-)^{\bullet}\) form a 2-equivalence.
\label{thm:2eqv-H}
\end{theorem}
\begin{proof}
We show that \(\two{\nu}_{(\cat{C},\alpha)}\) and \(\two{\eta}_{\cat{C}}\) provide the components for 2-natural isomorphisms
\begin{equation*}
\two{\nu} \colon ((-)^{\bullet})^1 \tonat \iid[\ModCat[H]] \qquad
\two{\eta} \colon ((-)^1)^{\bullet} \tonat \iid[\Cat^{H,\shifts}] \,.
\end{equation*}
We have already seen that the components are invertible.

First, let \((F,s) \colon (\cat{C},\alpha) \to (\cat{D},\beta)\) be an \(H\)-module functor.
Write \(\phi\) and \(\psi\) respectively for the induced shift functors on \((\cat{C},\alpha)^{\bullet}\) and \((\cat{D},\beta)^{\bullet}\).
By \cref{lem:shiftaction-functor,lem:2bullet-1}, \(((F,s)^{\bullet})^1\) maps objects \(X \mapsto FX\) and morphisms \(f \in \Hom_{((\cat{C},\alpha)^{\bullet})^1}(X,Y) = \Hom_{(\cat{C},\alpha)}(\alpha^1 X, Y)\) to \(Ff \of s^h_X\).
Recall from \cref{lem:2eqv-nu} the \(H\)-module equivalences \(\two{\nu}\); coherence for \(s\) gives
\begin{gather*}
\big( \two{\nu}_{(\cat{D},\beta)} \of ((F,s)^{\bullet})^1 \big) X
= FX
= \big( (F,s) \of \two{\nu}_{(\cat{C},\alpha)} \big) X
\\
\big( \two{\nu}_{(\cat{D},\beta)} \of ((F,s)^{\bullet})^1 \big) f
= Ff \of s^1_X \of e^{\alpha}_{FX}
= F(f \of \epsilon^{\alpha}_X)
= \big( (F,s) \of \two{\nu}_{(\cat{C},\alpha)} \big) f
\end{gather*}
so that \(\two{\nu}_{(\cat{C},\alpha)}\) is natural in \((\cat{C},\alpha)\).

Now let \(F \colon \cat{C} \to \cat{D}\) be an \(H\)-Hom-graded functor.
Suppose \(\cat{C}\) has shifts denoted \(r_{X,a} \colon X \to \phi^a X\) and \(\cat{D}\) has shifts denoted \(s_{Y,a} \colon Y \to \psi^a Y\).
Recall from \cref{lem:2eqv-eta} the \(H\)-Hom-graded equivalences \(\two{\eta}\); we have
\begin{gather*}
(\two{\eta}_{\cat{D}} \of (F^1)^{\bullet}) X
= \two{\eta}_{\cat{D}} (FX)
= FX
= (F \of \two{\eta}_{\cat{C}}) X
\\
(\two{\eta}_{\cat{D}} \of (F^1)^{\bullet}) f
= \two{\eta}_{\cat{D}} (F^1 f \of (s_F^h)_X)
= \big( Ff \of Fr^{\phi}_{X,h} \of (r^{\psi}_{FX,h})^{-1} \big) \of r^{\psi}_{FX,h}
= (F \of \two{\eta}_{\cat{C}}) f
\end{gather*}
so that \(\two{\eta}_{\cat{C}}\) is natural in \(\cat{C}\).
\end{proof}
\subsection{2-equivalence between \texorpdfstring{\(\Cat_{\tau}^{\shifts}\)}{Catτ<>} and \texorpdfstring{\(\ModCat[H]_{\tau}\)}{H-ModCatτ}}
\label{sec:2eqv2}
We obtain this final 2-equivalence as a restriction of \cref{thm:2eqv-H}, but first we need to identify the extension of \(\ModCat[H]\) corresponding to \(\Cat_{\tau}^{\shifts} \supseteq \Cat^{H,\shifts}\).

\begin{definition}
A \defined{\(\tau\)-module category} is an \(H\)-module category \((\cat{C},\alpha)\) such that
\begin{enumerate}
\item \(\cat{C}\) is \(G\)-graded with \(\cat{C} = \bigboxplus_{g \in G} \cat{C}_g\)
\item we have \(\alpha^h X \in \cat{C}_{\tau(h)g}\) for all \(h \in H\), \(g \in G\) and \(X \in \cat{C}_g\).
\end{enumerate}
We write \(\ModCat[H]_{\tau}\) for the 2-category of \(\tau\)-module categories, where
\begin{itemize}
\item \(\tau\)-module functors are \(H\)-module functors \((F,s)\) satisfying \(|FX|=|X|\) for all homogeneous objects \(X\)
\item \(\tau\)-module natural transformations are just \(H\)-module natural transformations between \(\tau\)-module functors.
\end{itemize}
\end{definition}

Writing \(\PSFun(\ccat{C},\ccat{D})\) for 2-categories of pseudofunctors as in \cite{jy2021} and recalling the groupoid \(\cat{G}_{\tau}\) from \cref{eg:tgrad:grpd}, a direct comparison of data gives
\begin{equation*}
\ModCat[H]_{\tau}
= \PSFun(\Disc(\cat{G}_{\tau}), \Cat)
\,.
\end{equation*}
When \(\tau = \chi\) is a crossed module, \(\cat{G}_{\chi}\) becomes a 2-group.
There is a standard notion of 2-group representation (see for example \cite{elgueta2007}), but this \textbf{does not} match the objects of \(\ModCat[H]_{\chi}\).
Representations of 2-groups are instead pseudofunctors from the delooping of \(\cat{G}_{\chi}\), and might be denoted
\begin{equation*}
\ModCat[\cat{G}_{\chi}]
= \PSFun(\deloop \cat{G}_{\chi}, \Cat)
\,.
\end{equation*}
In particular, \(\Disc(\cat{G}_{\chi})\) has objects from \(G\) and 1-morphisms from \(H\), whereas \(\deloop \cat{G}_{\chi}\) has 1-morphisms from \(G\) and 2-morphisms from \(H\).

\begin{theorem}
The 2-equivalence \((-)^1 \colon \Cat^{H,\shifts} \rightleftarrows \ModCat[H] \colon (-)^{\bullet}\) from \cref{thm:2eqv-H} extends to a 2-equivalence \(\Cat_{\tau}^{\shifts} \eqv \ModCat[H]_{\tau}\).
\label{thm:final2eqv}
\end{theorem}
\begin{proof}
Suppose \(\cat{C}\) is a \(\tau\)-graded category with shifts.
The \(H\)-module category \(\cat{C}^1\) inherits object degrees from \(\cat{C}\).
Since the \(H\)-action is precisely the shift functor, \(\cat{C}^1\) is an \(H\)-module category by \cref{rem:shiftdegree}.

In the other direction, suppose now that \((\cat{C},\alpha)\) is a \(\tau\)-module category.
The \(H\)-Hom-graded category \((\cat{C},\alpha)^{\bullet}\) again inherits object degrees from \((\cat{C},\alpha)\).
For homogeneous \(X\) and \(Y\) we have \(|\alpha^h X| = \tau(h) |X|\) so that
\begin{equation*}
\Hom_{(\cat{C},\alpha)^{\bullet}}^h(X,Y)
= \Hom_{(\cat{C},\alpha)}(\alpha^h X,Y)
= 0
\end{equation*}
unless \(|Y| = \tau(h) |X|\).

If \(F\) is a \(\tau\)-graded functor then we have \(|FX|=|X|\) automatically.
If \((F,s)\) is a \(\tau\)-module functor then we also have \(|(F,s)^{\bullet} X| = |FX| = |X|\).
There are no extra conditions on 2-morphisms.
\end{proof}

\begin{remark}
When \(G=1\), the indecomposable semisimple \(H\)-Hom-graded category \(\cat{M}_{\tau}(L,\psi)\) (from \cref{def:mtau}) corresponds under this 2-equivalence to the indecomposable \(H\)-module category in \cite[, Example 7.4.10]{egno2015}.
This is the precise sense in which \cref{thm:tgrad:struct} generalises the structure theorem for \(H\)-module categories.
\end{remark}
\section*{Acknowledgments}
\label{sec:orgb99c3a6}
\addcontentsline{toc}{section}{Acknowledgments}

This work was completed towards the author's PhD project, supported by a School Scholarship at the University of Nottingham, School of Mathematical Sciences.
The author wishes to thank Robert Laugwitz for his supervision and guidance.
\label{sec:orgffc4e4b}

\printbibliography[heading=bibintoc]
\end{document}